\documentclass[12pt]{amsart}

\usepackage[english]{babel} 

\usepackage{amsfonts, amssymb, amsmath, amsthm, latexsym}
\usepackage{geometry,graphicx}
\usepackage[colorlinks=true,linkcolor=blue,citecolor=blue]{hyperref}
\usepackage{enumerate}

\theoremstyle{plain}
\newtheorem{theorem}{Theorem}[section] 

\newtheorem{lemma}[theorem]{Lemma} 
\newtheorem{corollary}[theorem]{Corollary}
\newtheorem{proposition}[theorem]{Proposition}

\theoremstyle{definition} 

\numberwithin{equation}{section} 
\def\C{\mathbb C}

\def\H{\mathcal H}
\def\M{\mathcal M}
\def\S{\mathcal S}

\def\F{\mathcal F}
\def\Cont{\mathcal C}

\def\d{\,\mathrm{d}}
\def\dist{\mathrm{d}}
\def\fdist{\mathrm{d}_\phi}

\def\a{\alpha}
\def\e{\varepsilon}
\def\ee{\mathrm{e}}

\def\gFockh{F^2_\phi}
\def\gFock{F^p_\phi}

\def\Lpf{L^p_\phi}
\def\Pf{P_\phi}

\def\TT{{\mathcal T}_{\mu}}
\def\Tz{T_\mu}

\newcommand{\ef}[1]{\ee^{-\phi(#1)}}
\newcommand{\efh}[1]{\ee^{-2\phi(#1)}}

\newcommand{\tld}[1]{\widetilde{#1}}

\newcommand{\tr}[1]{\mathrm{tr}\left(#1\right)}

\newcommand{\cl}[1]{\overline{#1}}
\newcommand{\imply}[2]{$ $\vspace*{4pt} \newline \noindent #1 $\Rightarrow$ #2:}
\newcommand{\set}[2]{\left\{#1 \colon #2\right\}}
\newcommand{\psc}[1]{\langle #1 \rangle}
\newcommand{\fpsc}[1]{\langle #1 \rangle_\phi}
\newcommand{\pth}[1]{\left( #1 \right)}

\newcommand{\abs}[1]{\left| #1 \right|}
\newcommand{\norm}[1]{\left\| #1 \right\|}

\newcommand{\gfnorm}[1]{\norm{#1}_{p,\phi}}
\newcommand{\gfhnorm}[1]{\norm{#1}_{2,\phi}}
\newcommand{\gfinfnorm}[1]{\norm{#1}_{\infty,\phi}}
\newcommand{\infnorm}[1]{\norm{#1}_{\infty}}

\newcommand{\propref}[1]{Proposition~\ref{#1}}
\newcommand{\thmref}[1]{Theorem~\ref{#1}}
\newcommand{\corolref}[1]{Corollary~\ref{#1}}
\newcommand{\lemmaref}[1]{Lemma~\ref{#1}}

\newcommand{\secref}[1]{Section~\ref{#1}}

\DeclareMathOperator*{\esssup}{ess\,sup}

\title{Toeplitz Operators on Doubling Fock Spaces}
\author{Roc Oliver \and Daniel Pascuas}
\address{Departament de Matem\`atica Aplicada i An\`alisi, Universitat de Barcelona, Gran Via 585, 08007 Barcelona, Spain}
\email{roc.oliver@ub.edu}
\email[Corresponding author]{daniel\_pascuas@ub.edu}

\begin{document}

\bibstyle{amsplain}

\begin{abstract}
We study Toeplitz operator theory on the doubling Fock spaces, which are Fock spaces whose exponential weight is associated to a subharmonic function with doubling Riesz measure.	Namely,  we characterize the boundedness, compactness and membership in the Schatten class of Toeplitz operators on doubling Fock spaces whose symbol is a positive Radon measure.

\end{abstract}
\date{\today}
\thanks{Partially supported by  DGICYT Grants MTM2011-27932-C02-01 and MTM2014-51834-P, and DURSI Grant 2014SGR 289.}

\maketitle
\thispagestyle{empty}

\section{Introduction}

During the last decades many authors have contributed to develop an operator theory on the classical Fock spaces (see the recent book \cite{zhu2012fock} for an account of that theory). We are interested in a more general setting of Fock spaces, the doubling Fock spaces, which we are going to introduce. Denote by $\dist A$ the Lebesgue area measure on the complex plane $\C$, and let $\H(\C)$ be the space of entire functions.  Let $\phi$ be a subharmonic function on $\C$. Then $\Delta \phi$ is a locally finite positive Borel measure. From now on we suppose that $\Delta \phi$ is a doubling measure. For $0<p <\infty$, let $L^p_{\phi} := L^p_{\phi}(\C)$ be the space of all measurable functions $f$ on $\C$ such that
$$\norm{f}_{p,\phi}^p := \int_\C \abs{f(z) \ee^{-\phi(z)}}^p \d A(z) < \infty,$$
and 
$$L^\infty_{\phi} = L^\infty_\phi(\C) := \set{f \ \mathrm{meas.}}{\gfinfnorm{f} := \esssup_{z\in\C} \abs{f(z) \ee^{-\phi(z)}} < \infty}.$$
It is clear that $L^p_\phi = L^{p}_{}(\C, \ee^{-p \phi}\! \d A)$, for $0 < p < \infty$. Moreover, $(L^p_\phi, \norm{\,\cdot\,}_{p,\phi})$ is a Banach space for $1 \leq p \leq \infty$, and a quasi-Banach space for $0<p<1$. The doubling Fock spaces $\gFock$ are defined to be
$$
\gFock := \H(\C) \cap \Lpf \qquad (0 < p \leq \infty).
$$
Recall that for $\phi(z) = \frac{\a}{2} \abs{z}^2$, where $\alpha>0$,  one obtains the classical Fock spaces. Moreover, if $\phi$ is a subharmonic function on $\C$ such that $\Delta\phi$ is comparable to the Lebesgue measure $\dist A$,
 then $\Delta\phi$ is a doubling measure. But there are subharmonic functions $\phi$ on $\C$ such that $\Delta\phi$ is a doubling measure which is not comparable to $\dist A$, for example, any subharmonic non-harmonic polynomial on $\C$ with degree greater than~{$2$.}
 
 As far as we know, the first work dealing with such spaces is the 
 seminal paper~{\cite{marco2003interpolating}} by N. Marco, X. Massaneda and J. Ortega-Cerd\`{a}, where the authors characterized the interpolating and sampling sequences for the doubling Fock spaces by extending the corresponding characterizations due to J. Ortega-Cerd\`{a} and K. Seip for the case that $\Delta\phi$ is comparable to the Lebesgue measure (see~{\cite{ortegaseip}}).
 Moreover, in doing so they proved a large amount of technical properties of the doubling Fock spaces. 
 Furthermore, J. Marzo and J. Ortega-Cerd\`{a} completed the previous work in their interesting paper~{\cite{marzo2009pointwise}} by showing quite sharp pointwise estimates of the
 Bergman kernel associated to these spaces.
 
The goal of this article is to apply all this technical machinery to study the theory of Toeplitz operators acting on the doubling Fock spaces.
But previously we obtain some complementary results such as the complex interpolation and the duality of these spaces. Namely, we completely characterize the boundedness, compactness and membership in the Schatten class of the Toeplitz operators on the doubling Fock spaces whose symbol is a positive locally finite Borel measure.  In doing so we prove characterizations of the so-called Carleson measures and vanishing Carleson measures associated to these spaces. 
All the above characterizations are in terms of the Berezin and average transforms of the symbol measure.
Our results extend to the setting of the doubling Fock spaces previous known results for the classical Fock spaces (see~{\cite{zhu2012fock}}) and the case that $\Delta\phi$ is comparable to the Lebesgue measure
(see~{\cite{isra2010schatten, xiao-wang-xia, schuster}}).
 For similar results in the setting of the Bergman spaces on the unit disk 
 see~{\cite{hicham1,hicham2}}.
 
 The paper is organized as follows. In the next section we will fix the notation and introduce some technical but useful properties of our Fock spaces. \secref{sec:proj} deals with the Bergman projection on the doubling Fock  spaces and its interesting consequences, which are the duality and the complex interpolation of these spaces. 
 In \secref{sec:fock-carleson:measures} we characterize the so-called Fock-Carleson measures and  vanishing Fock-Carleson measures.
 In \secref{sec:toeplitz} we introduce the Toeplitz operators on the doubling Fock spaces whose symbols are locally finite positive Borel measures, and 
 we characterize their boundedness and compactness in terms of their symbols.
 Finally, a complete description of the membership in the Schatten class of those Toeplitz operators is given in \secref{sec:membership:shatten:class}, which is the last section of the paper and the one that contains the more elaborate proofs of the paper.
 
Finally a word about notation. We write either $f \lesssim g$ or $g\gtrsim f$ whenever there is a positive constant $C$, independent of the variables involved, such that $f \leq C g$, and $f\simeq g$ if both $f\lesssim g$ and $g \lesssim f$ hold.

%
%
%
%

\section{Basic Properties}

As usual, we denote by $D(z,r)$ the open disk in $\C$ of center $z\in \C$ and radius $r>0$. A positive Borel measure $\mu$ on $\C$ is called \emph{doubling} if there exists a constant $C>1$ such that 
\begin{equation}\label{doubling:measure}
0<\mu(D(z,2r)) \leq C \mu(D(z,r))<\infty,\quad\mbox{for every $z\in \C$ and $r>0$.}
\end{equation}
 The smallest constant $C>1$ satisfying \eqref{doubling:measure} is called the \emph{doubling constant}~{for $\mu$} and  is denoted by $C_{\mu}$. Note that then
 $\mu(D(z,2r))\ge c \mu(D(z,r))$, for every $z\in\C$ and $r>0$,
 where $c=1+C^{-3}_{\mu}>1$, and therefore 
 \begin{equation}\label{infinite:mass}
 \lim_{r\to\infty}\mu(D(z,r))=\infty,\quad\mbox{for every $z\in \C$.}
 \end{equation}
Moreover, it is well known that $\mu$ has no mass on any circle (see \cite[p.\ 40]{stein}), and,
in particular, $\mu$ has no atoms, that is,
\begin{equation}\label{zero:mass}
\mu(\partial D(z,r))=\mu(\{z\})=0,\quad\mbox{for every $z\in \C$ and $r>0$.}
\end{equation}

An important and useful estimate for doubling measures is the following result due to M. Christ:

\begin{lemma}[{\cite[Lemma 2.1]{christ1991}}] \label{estimate:doubling}
	Let $\mu$ be a doubling measure on $\C$. Then there are constants $C>1$ and $0<\delta<1$, which only depend on $C_{\mu}$, such that
	 if $D$ and $D'$ are open disks of radii $r$ and $r'$, respectively, such that
	 $D\cap D'\ne\emptyset$ and $r'<r$, then 
	 $$
	 C^{-1}(r'/r)^{1/\delta}\mu(D)\le\mu(D')\le C(r'/r)^{\delta}\mu(D).               
	 $$
\end{lemma}

Let $\phi$ be a subharmonic function on $\C$ such that $\mu = \Delta \phi$ is a doubling measure.   Then $\mu$ satisfies \eqref{infinite:mass} and \eqref{zero:mass}
 so the map $r\in(0,\infty)\mapsto\mu(D(z,r))\in(0,\infty)$ is a strictly increasing homeomorphism. In particular,  
 for every $z\in \C$ there is a unique radius $\rho(z)=\rho_{\phi}(z)>0$ such that $\mu(D(z,\rho(z)))=1$.
 Note that, for $z,w\in\C$, we have that 
 $D(z,\rho(z))\subset D(w,\rho(z)+|z-w|)$, so  
 $1\le \mu(D(w,\rho(z)+|z-w|))$, and therefore $\rho(w)\le\rho(z)+|z-w|$. By symmetry it follows that
 \begin{equation}\label{rho:Lipschitz}
 |\rho(w)-\rho(z)|\le|z-w|,\quad\mbox{ for every $z,w\in\C$.} 
 \end{equation} 
 
The function $\rho^{-2}$ can be considered as a regularization of the measure $\Delta \phi$. Indeed, there exist $\psi \in \Cont^\infty(\C)$ and a constant $C>0$ such that $|\phi-\psi|\leq C$, 
 $(\Delta \psi)\d A$ is a doubling measure and
	$$\Delta \psi \simeq \frac{1}{\rho^2_\psi} \simeq \frac{1}{\rho^2_\phi}.$$
Due to that fact, the space $\gFock$ does not change if $\phi$ is replaced by $\psi$, so from now on we will assume that $\phi \in \Cont^\infty(\C)$ and $\Delta \phi \simeq 1/\rho^2$. For this reason, we will use sometimes the notation $\!\d \sigma := \!\d A/\rho^2$. See~\cite{marco2003interpolating} for all that.

Let $D^r(z) := D(z, r\rho(z))$ and $D(z): = D^1(z)=D(z,\rho(z))$, for $z\in\C$ and $r>0$.  We also use the following notations: $D^r(z)^c := \C \setminus D^r(z)$ and $D(z,r)^c := \C\setminus D(z,r)$. 

As a consequence of  \eqref{rho:Lipschitz} and \lemmaref{estimate:doubling} we obtain the following useful estimate:

\begin{lemma}\label{lemma:rhoinballs}
	For every $r>0$ there is a constant $c_r\ge1$, depending only on $r$ and  the doubling constant for $\Delta \phi$, such that
	\begin{equation}\label{eq:rhoinballs}
	c_r^{-1}\,\rho(z) \leq \rho(w) \leq c_r\,\rho(z), 
	\quad \mbox{ for every $z\in \C$ and  $w\in D^r(z)$.}
	\end{equation}
	Namely, $c_r=(1-r)^{-1}$, for every $0<r<1$.
\end{lemma}

\begin{proof}
Observe that~{\eqref{rho:Lipschitz}} shows that
$$
(1-r)\rho(z)\le\rho(w)\le(1+r)\rho(z),
\quad \mbox{ for every $z\in\C$ and $w\in D^r(z)$.}
$$
Therefore $c_r=(1-r)^{-1}$ satisfies \eqref{eq:rhoinballs}, for every $0<r<1$, and we only have to prove the  first estimate of \eqref{eq:rhoinballs} for $r>1$ and $w\in D^r(z)$ such that $\rho(w)<\rho(z)$. In this case, we may apply \lemmaref{estimate:doubling}  to the doubling measure $\mu=\Delta\phi$ and the disks $D=D^r(z)$ and $D'=D(z)$, so we get 
$$
\mu(D^r(z))\le Cr^{\frac1{\delta}},
$$
where $C>1$ and $\delta\in(0,1)$ are constants depending only on $C_{\mu}$.   
Then another application of  \lemmaref{estimate:doubling} to the disks $D=D^r(z)$ and $D'=D(w)$ shows that   
$$
1\le C\left(\frac{\rho(w)}{r\rho(z)}\right)^{\delta}\mu(D^r(z))
  \le C^2r^{\frac1{\delta}-\delta}\left(\frac{\rho(w)}{\rho(z)}\right)^{\delta},
$$
and the proof is complete.
\end{proof}
	
The behavior of $\rho$ outside the disks $D(z)$ also follows from \lemmaref{estimate:doubling} as M. Christ proved:

\begin{lemma}[{\cite[Lemma 3.3]{christ1991}}]  \label{lemma:rhooutsideballs}
There is a constant $\delta \in (0,1)$, depending only on the doubling constant for $\Delta \phi$, such that
$$
\rho(z) \lesssim \abs{z-w}^{1-\delta} \rho(w)^{\delta} 
\qquad (z\in\C, w\in D(z)^c)
$$
and
$$
\rho(w) \lesssim \abs{z-w}^{1-\delta}\rho(z)^\delta
 \qquad (z\in\C, w\in D(z)^c).
$$
\end{lemma}

We continue with a useful and well-known result that is widely used in many situations throughout this work.

\begin{lemma}[{\cite[Lemma 19(a)]{marco2003interpolating}}] \label{lemma:estimates}
	Let $0< p < \infty$. For any $r>0$ there exists a constant $C > 0$ such that, for any $f\in \H(\C)$ and $z\in \C$,
	$$
	\abs{f(z) \ef{z}}^p \leq C \int_{D^r(z)} \abs{f(w) \ef{w}}^p \dfrac{\d A(w)}{\rho(w)^2}.
	$$
\end{lemma}

It is a direct consequence of \lemmaref{lemma:estimates} and \eqref{eq:rhoinballs} that, for any $0< p \leq \infty$, we have
	\begin{equation} \label{eq:est}
		\abs{f(z)} \lesssim \frac{\ee^{\phi(z)}}{\rho(z)^{2/p}} \gfnorm{f} \qquad (z\in \C, f\in \gFock).
	\end{equation}
This pointwise estimate implies that $(\gFock, \norm{\, \cdot\,}_{p,\phi})$ is a Banach space for $1\leq p \leq \infty$, and a quasi-Banach space for $0 < p <1$. Another consequence is that, for any $z \in \C$, the pointwise evaluation $f \mapsto f(z)$ is a bounded linear functional on $\gFock$. In particular, $\gFockh$ is a reproducing kernel Hilbert space: there exists a unique function $K_z$ in $\gFockh$ such that $f(z) = \psc{f, K_z}_\phi$, for every $f\in \gFockh$, where
\begin{equation}\label{eq:fpsc}
	\psc{f, g}_\phi := \int_\C f(w) \overline{g(w)} \efh{w} \d A(w)
\end{equation}
is the inner product in $L^2_\phi$. In particular, $K_z(w) = \psc{K_z, K_w}_\phi = \overline{K_w(z)}$, for every $z,w\in\C$.
The function $K_z$ is called the \emph{reproducing kernel} or \emph{Bergman kernel} for $\gFockh$ at $z\in \C$.

First, we recall  the $L^2_{\phi}$-norm estimate of the Bergman kernels (see~{\cite[{Prop.}~{2.10}]{marzo2009pointwise}}):
\begin{equation} \label{prop:kerneldiagonal}
\gfhnorm{K_z} \simeq \frac{\ee^{\phi(z)}}{\rho(z)} \qquad (z\in\C).
\end{equation}

In order to state the pointwise estimates of the Bergman kernels it is convenient to consider the distance $\dist_\phi$ induced by the metric $\rho^{-2} (z) \d z \otimes \!\d \cl{z}$. Namely, for any $z, w\in \C$,
$$\dist_\phi(z,w) := \inf_\gamma \int_0^1 \abs{\gamma' (t)} \frac{\d t}{\rho(\gamma(t))},$$
where $\gamma$ runs on the piecewise $\Cont^1$ curves $\gamma \colon [0,1] \to \C$ with $\gamma(0) = z$ and $\gamma(1) = w$. 
This distance satisfies the following estimates:

\begin{lemma}[{\cite[Lemma 4]{marco2003interpolating}}] 
	For every $r>0$ there is a constant $c_r>1$ such that
	\begin{equation}\label{dist:estimate1}
	c_r^{-1}\, \frac{|z-w|}{\rho(z)} \le \fdist(z,w)\le c_r\, \frac{|z-w|}{\rho(z)}\,\quad(z\in\C,\,w\in D^r(z))
	\end{equation}
	and
	\begin{equation}\label{dist:estimate2}
	c_r^{-1} \pth{\frac{|z-w|}{\rho(z)}}^{\delta} \leq \fdist(z,w)\le c_r \pth{\frac{|z-w|}{\rho(z)}}^{1/\delta}
	\quad(z\in\C,\,w\in D^r(z)^c),
	\end{equation}
	where $\delta\in(0,1)$ is the constant in \lemmaref{lemma:rhooutsideballs}.
\end{lemma}

Now we can state the pointwise estimates of the Bergman kernel.

\begin{theorem}[{\cite[Theorem 1.1, (3) and Proposition 2.11]{marzo2009pointwise}}]
	There exist constants $C >0$ and $\e >0$ (depending only on the doubling constant for $\Delta \phi$) such that
	\begin{equation} \label{eq:kernel1}
		\abs{K_z(w)} \leq C \frac{1}{\rho(z) \rho(w)}\frac{\ee^{\phi(z) + \phi(w)}}{\exp{\pth{\!\d_\phi(z, w)^\e}}}, \qquad \textrm{for any\, } z,w\in\C.
	\end{equation}
	Moreover, there is $r_0>0$ such that
	\begin{equation}\label{eq:kernel2} 
		\abs{K_z(w)} \simeq \gfhnorm{K_z} \gfhnorm{K_w} \simeq \frac{\ee^{\phi(z) + \phi(w)}}{\rho(z) \rho(w)}, \qquad (z \in\C, w \in D^{r_0}(z)).
      \end{equation}
\end{theorem}

The following two lemmas are very useful to prove the $L^p_{\phi}$-norm estimates of the Bergman kernels.

\begin{lemma} \label{lemma:integral:estimate}
	For every $\varepsilon>0$, $k\ge0$ and $r\ge1$ there is a constant $C_{\varepsilon,k}(r)>0$ such that 
$$
\int_{D^r(z)^c}\frac{|w-z|^k}{\exp{\pth{\!\d_\phi(w,z)^{\varepsilon}}}}\d\sigma(w)
\le C_{\varepsilon,k}(r)\,\rho(z)^k,\quad\mbox{for every $z\in\C$.}
$$
Moreover, $ C_{\varepsilon,k}(r)\to 0$, as $r\to\infty$, for any  $\varepsilon>0$ and $k\ge0$.
\end{lemma}

Lemma \ref{lemma:integral:estimate} is easily proved by following the proof of   
\cite[Lemma 2.7]{marzo2009pointwise}.

\begin{lemma} \label{lemma:kernel}\hspace*{\fill}
\begin{enumerate}[(a)]
		\vspace*{3pt}
		\item For every $r\ge1$ there is a constant $C(r)>0$ such that
$$
      \int_{D^r(z)^c} \abs{ K_z(w) \ef{w}} \d A(w) \leq C(r)  \ee^{\phi(z)}, \quad \textrm{for every\ } z\in\C,
$$
and $C(r)\to 0$, as $r\to\infty$.
		\vspace*{3pt}
		\item There exists a constant $C>0$ such that 
$$
      \int_\C \abs{ K_z(w) \ef{w}} \d A(w) \leq C  \ee^{\phi(z)}, \quad \textrm{for every\ } z\in\C.
$$
	\end{enumerate}	
\end{lemma}

\begin{proof}\hspace*{\fill}\vspace*{3pt}

\noindent
(a)
By \eqref{eq:kernel1} there is $\e >0$ such that
$$
\int_{D^r(z)^c} \abs{K_z(w) \ef{w}} \d A(w) \lesssim  \frac{\ee^{\phi(z)} }{\rho(z)} \int_{D^r(z)^c} \frac{\rho(w)\d\sigma(w)}{\exp\pth{ \dist_\phi (z,w)^\e}}\quad(z\in\C,\,r>0)
$$
	Then, by \lemmaref{lemma:rhooutsideballs}, there is $\delta\in(0,1)$ such that
$$
\frac{e^{\phi(z)}}{\rho(z)} \frac{\rho(w)}{\exp\pth{ \dist_\phi (z,w)^\e}}
\lesssim  \frac{\ee^{\phi(z)}}{\rho(z)^{1-\delta}} \frac{ \abs{z-w}^{1-\delta}}{\exp\pth{ \dist_\phi (z,w)^\e}}
\quad(z\in\C,\,r\ge1,\,w\in D^r(z)).
$$
Therefore \lemmaref{lemma:integral:estimate} shows that
$$
\int_{D^r(z)^c} \abs{K_z(w) \ef{w}} \d A(w) \lesssim  C_{\varepsilon,1-\delta}(r)\,\ee^{\phi(z)}
\quad(z\in\C,\,r\ge1),
$$
so (a) holds.
\vspace*{3pt}

\noindent
(b)	By (a) we have that
$$
      \int_{D(z)^c} \abs{ K_z(w) \ef{w}} \d A(w) \lesssim\ee^{\phi(z)}\quad(z\in\C).
$$
On the other hand, since $\abs{K_z(w)} \le  \gfhnorm{K_z} \gfhnorm{K_w}$, \eqref{prop:kerneldiagonal} and  \eqref{eq:rhoinballs} show that
$$
\int_{D(z)}  \abs{K_z(w) \ef{w}} \d A(w)
\lesssim\frac{\ee^{\phi(z)}}{\rho(z)} \int_{D(z)} \frac{\d A(w)}{\rho(w)}  \simeq \ee^{\phi(z)}\quad(z\in\C ).
$$
	And the proof is complete.
\end{proof}

\begin{proposition} \label{prop:normkz}
For any $1\leq p \leq \infty$ we have that
$$
	\gfnorm{K_z} \simeq \ee^{\phi(z)} \rho(z)^{\frac{2}{p}-2} \qquad (z\in \C).
$$
\end{proposition}

\begin{proof}
First note that the estimate $\gtrsim$ directly follows from \eqref{eq:kernel2} and \eqref{eq:rhoinballs}, so let us prove the opposite estimate. For $p=1$ it is just \lemmaref{lemma:kernel}(b). For $p=\infty$ it follows from \lemmaref{lemma:estimates}, \eqref{eq:rhoinballs} and \lemmaref{lemma:kernel}(b):
\begin{align*}
	\abs{K_z(w)\ef{w}} =\ef{w}\ee^{\phi(z)} \abs{K_w(z)\ef{z}} \lesssim \frac{\ef{w}\ee^{\phi(z)}}{\rho(z)^2} \int_{D(z)} \abs{K_w \ee^{-\phi}} \d A \lesssim  \frac{\ee^{\phi(z)}}{\rho(z)^2}.
\end{align*}
Finally, the case $1<p<\infty$ is a direct consequence of the two preceding cases:
\begin{align*}
	\int_\C \abs{K_z(w)\ef{w}}^p \d A(w) & = \int_\C \abs{K_z(w)\ef{w}} \abs{K_z(w)\ef{w}}^{p-1} \d A(w)\\
	& \lesssim \ee^{(p-1)\phi(z)}\rho(z)^{2-2p}\int_\C \abs{K_z(w)\ef{w}} \d A(w) \\
	& \lesssim  \ee^{p\phi(z)}\rho(z)^{2-2p}. \qedhere
\end{align*}
\end{proof}

%
%
%
%

\section{Bergman Projection, Duality and Complex Interpolation} \label{sec:proj}

It is straightforward to see from the reproducing property of the Bergman kernel for $\gFockh$ that the orthogonal projection $P_\phi \colon L^2_\phi \to \gFockh$ is the integral operator given by
$$P_\phi f(z) := \int_\C f(w) \overline{K_z(w)} \ee^{-2\phi(w)} \d A(w) \qquad (f\in L^2_\phi, z\in\C).$$

\begin{theorem} \label{thm:projbnd}
	$\Pf$ is a bounded linear operator from $\Lpf$ to $\gFock$, for any $1 \leq p \leq \infty$.
\end{theorem}

\begin{proof}
First we prove that $\Pf f \in \H(\C)$, for every $f\in\Lpf$.
We do that by differentiation under the integral sign. 
Let $f\in\Lpf$. 
Since the function
$$
F(z,w):=f(w)\overline{K_z(w)}\ee^{-2\phi(w)}\qquad(z,w\in\C)
$$ 	
satisfies that $F(\cdot,w)=f(w)K_w\ee^{-2\phi(w)}$ is an entire function, for every $w\in\C$, and $F(z,\cdot)=f\overline{K}_z\ee^{-2\phi}$ is a continuous function on $\C$, for every $z\in\C$, we only have to check that
for every $z_0\in\C$ there is $G_{z_0}\in L^1(\C)$ such that
\begin{equation}\label{locally:uniform:domination}
|F(z,w)|\lesssim G_{z_0}(w)\qquad(z\in D(z_0),\,w\in\C).
\end{equation}
But \lemmaref{lemma:rhoinballs} and the 
 subharmonicity of $|K_w|$ imply that there is a constant $r>1$ such that the function
$$
G_{z_0}(w):=
|f(w)|e^{-2\phi(w)}\int_{D^r(z_0)}|K_w(\zeta)|\d A(\zeta)
\qquad(w\in\C)
$$ 
satisfies~{\eqref{locally:uniform:domination}}.  Moreover,
 H\"{o}lder's inequality, Jensen's formula, Tonelli's theorem  and \propref{prop:normkz} show that $G_{z_0}\in L^1(\C)$.

Now we are going to prove that $P_{\phi}$ is a bounded linear 
operator on $\Lpf$. 
	 Observe that $\abs{\Pf f(z) \ef{z}} \leq Q_\phi(\abs{f\ee^{-\phi}})(z)$, for every $z\in\C$, where
		$$Q_\phi g(z) := \int_\C g(w) J(z,w) \d A(w) \qquad (z\in\C),$$
	and $J(z,w):= \abs{K_z(w)} \ef{z} \ef{w}$. In particular, it turns out that $\Pf$ is bounded on $\Lpf$ whenever $Q_\phi$ is bounded on $L^p(\C)$. First of all, we know by \lemmaref{lemma:kernel}(b) that there exists $C>0$ such that $\int_\C J(z,w)\d A(z) \leq C$, for every $w\in\C$. Then, if $1< p < \infty$ and $q$ is the conjugate exponent of $p$, by H\"older's inequality and Fubini's theorem we have
	\begin{align*}
		\int_\C \abs{Q_\phi g(z)}^p\d A(z) & \leq C^{p/q} \int_\C \int_\C \abs{g(w)}^p J(z,w) \d A(w) \d A(z) \\
		& = C^{p/q} \int_\C \abs{g(w)}^p \int_\C J(z,w) \d A(z)\d A(w) \\
		& \leq C^{p} \int_\C \abs{g(w)}^p \d A(w).
	\end{align*}
	For $p=1$, Fubini's theorem shows that
	\begin{align*}
		\norm{Q_\phi g}_{1} & \leq \int_\C  \abs{g(w)} \int_\C J(z,w) \d A(z)\d A(w) \leq C \norm{g}_{1}.
	\end{align*}
	Finally, if $g \in L^\infty(\C)$ then
	\begin{align*}
		\abs{Q_\phi g(z)} & \leq \norm{g}_\infty \int_\C J(z,w)\d A(w) \leq C \norm{g}_\infty, \qquad \textrm{for every\ } z\in\C.
	\end{align*}
	Hence the proof is complete.
\end{proof}

\begin{corollary} \label{corol:Pselfadjoint}
	If $1\leq p \leq \infty$ and $q$ is the conjugate exponent of $p$, then
	$$\fpsc{\Pf f, g} = \fpsc{f, \Pf g}, \qquad \textrm{for every\, } f\in \Lpf \textrm{\, and\, } g\in L^q_\phi.$$
\end{corollary}

\begin{proof}
	It follows from Fubini's theorem. Note that the hypothesis of Fubini's theorem holds due to H\"older's inequality and the $L^p$-boundedness of the operator $Q_\phi$.
\end{proof}

\begin{theorem} \label{thm:integrep}
	Let $1\leq p \leq \infty$. Then $f = \Pf f$, for every $f\in \gFock$.
\end{theorem}

The proof of \thmref{thm:integrep} follows the approach of Lindholm (see \cite[pp.~412-413]{lindholm}). First we need the following approximation lemma. 

\begin{lemma} \label{lemma:approx}
	Let $1\leq p \leq \infty$. For any $f\in\gFock$ there is a sequence $\{f_n\}_n$ of functions in $\gFockh \cap \gFock$ such that: 
	\begin{enumerate}[(a)]
		\vspace*{3pt}
		\item $\displaystyle \lim_{n\to\infty} \norm{f_n - f}_{p,\phi} = 0$, if $p < \infty$.
		\vspace*{3pt}
		\item $\displaystyle\sup_{n\geq 1} \norm{f_n-f}_{\infty,\phi} < \infty$, and $\displaystyle \lim_{n\to\infty} \abs{f_n(z) - f(z)}\ef{z} = 0$, for every $z\in\C$, if $p=\infty$.
	\end{enumerate}
\end{lemma}

\begin{proof}
	Let $f\in \gFock$. Take a cutoff function $\psi \in \Cont^\infty(\C)$ such that $\psi(z) = 1$ if $\abs{z} \leq 1$, $0 < \psi(z) < 1$ if $1<\abs{z}<2$, and $\psi(z) = 0$ if $\abs{z} \geq 2$. Let $\psi_n(z) := \psi(z/n)$ and $f_n := \Pf (\psi_n f)$. Then $u_n := \psi_n f - f_n$
	is the $L^2_\phi$-minimal solution to the equation $\overline{\partial} u = f\overline{\partial} \psi_n$. By \cite[Proposition 1.4]{marzo2009pointwise},
	$$\norm{u_n}_{p,\phi} \lesssim \norm{\rho f \overline{\partial}\psi_n}_{p,\phi} \qquad (n\geq 1).$$
	Since 
	$$\overline{\partial} \psi_n(z) =  \frac{1}{n}\; (\overline{\partial} \psi)\pth{\frac{z}{n}},$$
	\eqref{eq:rhoinballs} and \lemmaref{lemma:rhooutsideballs} show that
	$$\abs{\rho(z) \overline{\partial}\psi_n(z) } \lesssim \frac{\abs{z}^{1-\delta}}{n} \, \chi_{D(0,2n)\setminus D(0,n)}(z) \lesssim \frac{1}{n^\delta} \qquad (n\geq 1, z\in \C),$$
	and so
	$$\norm{u_n}_{p,\phi} \lesssim \norm{\rho f \overline{\partial}\psi_n}_{p,\phi} \lesssim \frac{1}{n^\delta} \norm{f}_{p,\phi} \qquad (n\geq 1).$$
	Therefore
	$$\norm{f-f_n}_{p,\phi} \leq \norm{f - \psi_n f}_{p,\phi} + \norm{u_n}_{p,\phi}
		\lesssim \norm{f - \psi_n f}_{p,\phi} + \frac{1}{n^\delta}\norm{f}_{p,\phi}\quad (n\geq 1).$$
		
	\noindent (a) If $p<\infty$ then $\norm{f - \psi_n f}_{p,\phi} \to 0$ and so $\norm{f-f_n}_{p,\phi} \to 0$.
	
	\vspace*{3pt}
	\noindent (b) If $p=\infty$ then $\norm{f - \psi_n f}_{\infty,\phi} \leq \norm{f}_{\infty, \phi}$ and so the first assertion of (b) holds. The second assertion follows from the estimate
	\begin{align*}
		\abs{f(z) - f_n(z)} \ef{z} & \leq \abs{f(z) -\psi_n(z) f(z)} \ef{z} + \abs{u_n(z)} \ef{z} \\
		& \lesssim \abs{f(z) -\psi_n(z) f(z)} \ef{z} + \frac{1}{n^\delta}\norm{f}_{\infty,\phi},
	\end{align*}
	since $\abs{f(z) -\psi_n(z) f(z)} \ef{z} \to 0$.
\end{proof}

\begin{proof}[Proof of \thmref{thm:integrep}]
	Let $f\in\gFock$. Take $\{f_n\}_n$ as in \lemmaref{lemma:approx}. Recall that $f_n \in \gFockh$ and so $\Pf f_n = f_n$. We distinguish the following two cases:
	\vspace*{4pt}
	
	\noindent Case 1: $p < \infty$. Since $\Pf$ is bounded on $\Lpf$ (see \thmref{thm:projbnd}) and $ \norm{f_n - f}_{p,\phi} \to 0$, we have that $\norm{\Pf f_n - \Pf f}_{p,\phi} \to 0$. Therefore
	$$\Pf f(z) = \lim_{n\to\infty} \Pf f_n(z) = \lim_{n\to\infty} f_n(z) = f(z), \quad \textrm{for every\ } z\in\C.$$

	\noindent Case 2: $p = \infty$. Fix $z\in\C$. First note that 
	$$\abs{\Pf f(z) - \Pf f_n(z)} \leq \int_\C \abs{f(w)-f_n(w)}\ef{w} \abs{K_z(w)} \ef{w} \d A(w) \longrightarrow 0,$$
	by the dominated convergence theorem, since $K_z \in L^1_\phi$ and \lemmaref{lemma:approx}(b). Then
	$$\abs{\Pf f(z) -f(z)} \leq \abs{\Pf f(z) - \Pf f_n(z)} + \abs{f_n(z) - f(z)} \longrightarrow 0,$$
	and hence $\Pf f(z) = f(z)$.
\end{proof}

Observe that Theorems~\ref{thm:projbnd} and \ref{thm:integrep} show that $\Pf$ is a bounded projection of $\Lpf$ onto $\gFock$ (that is, $\Pf \colon \Lpf \to \Lpf$ is a bounded linear operator such that $\Pf \circ \Pf = \Pf$ and $\Pf \Lpf = \gFock$), for any $1\leq p \leq \infty$. As a consequence of that fact we will obtain results on complex interpolation and duality of the generalized Fock spaces. 

\begin{theorem}
	Let $1\leq p_0 \leq p_1 \leq \infty$ and $0\leq \theta \leq 1$. Then $[L^{p_0}_\phi, L^{p_1}_\phi]_\theta = L^{p_\theta}_\phi$ and $[F^{p_0}_\phi, F^{p_1}_\phi]_\theta = F^{p_\theta}_\phi$ with equivalent norms, where
	$$\frac{1}{p_\theta} = \frac{1-\theta}{p_0} + \frac{\theta}{p_1}.$$
\end{theorem}

\begin{proof}
	Let $1\leq p \leq \infty$. Then the formulas $Ef = f\ee^{-\phi}$ and $Rg = g\ee^{\phi}$ define bounded linear operators $E\colon \Lpf \to L^p(\C)$ and $R\colon L^p(\C) \to \Lpf$. (Indeed, $E$ and $R$ are topological isomorphisms.) Moreover, they satisfy $R \circ E = I$, where $I$ is the identity map on $\Lpf$. So, using the terminology of \cite[p.~151]{kalton2007interpolation}, $\{L^{p_0}_\phi, L^{p_1}_\phi\}$ is a retract of $\{L^{p_0}(\C), L^{p_1}(\C) \}$, and therefore the first part of \cite[Lemma 7.11]{kalton2007interpolation} shows that
	$$[L^{p_0}_\phi, L^{p_1}_\phi]_\theta = R([L^{p_0}(\C), L^{p_1}(\C)]_\theta) = R(L^{p_\theta}(\C)) = L^{p_\theta}_\phi$$
	with equivalent norms. Recall that the norms of $R(L^{p_\theta}(\C))$ (see \cite[(7.40)]{kalton2007interpolation}) and $L^{p_\theta}_\phi$ are equivalent because $R$ is a topological isomorphism. 

On the other hand, as we observed above, by Theorems \ref{thm:projbnd} and \ref{thm:integrep} we know that $\Pf$ is a bounded projection of $\Lpf$ onto $\gFock$. Then the second part of \cite[Lemma 7.11]{kalton2007interpolation} implies that
	$$[F^{p_0}_\phi, F^{p_1}_\phi]_\theta = [\Pf L^{p_0}_\phi, \Pf L^{p_1}_\phi]_\theta= \Pf( [L^{p_0}_\phi, L^{p_1}_\phi]_\theta) = \Pf L^{p_\theta}_\phi = F^{p_\theta}_\phi$$
	with equivalent norms. Note that the norms of $\Pf L^{p_\theta}_\phi$ (see \cite[(7.40)]{kalton2007interpolation}) and $F^{p_\theta}_\phi$ are equivalent by the open mapping theorem.
\end{proof}

It is well-known that if $1\leq p < \infty$ and $q$ is the conjugate exponent of $p$ then $(\Lpf)^*$ can be (isometrically) identified with $L_\phi^q$ by means of the integral pairing $\psc{\,\cdot\, ,\cdot\,}_\phi$ defined by \eqref{eq:fpsc}. Namely, the mapping
\begin{equation}\label{dual:Lp}
g\in L^q_{\phi} \longmapsto  \langle\,\cdot\,,g\rangle_{\phi}\in(L^p_{\phi})^{*}
\end{equation}
is an isometric antilinear isomorphism.
 From this fact and the boundedness of the projection $\Pf$ we are able to describe the dual of  $\gFock$, for $1\leq p <\infty$.

\begin{theorem} \label{thm:dual}
	Let $1\leq p < \infty$ and let $q$ be the conjugate exponent of $p$. Then $(\gFock)^*$ can be identified with $F^q_\phi$ (with equivalent norms) by means of the integral pairing $\psc{\,\cdot\, ,\cdot\,}_\phi$ given by \eqref{eq:fpsc}. Namely, the mapping
\begin{equation}\label{dual:Fp}
g\in F^q_{\phi} \longmapsto  \langle\,\cdot\,,g\rangle_{\phi}\in(F^p_{\phi})^{*}
\end{equation}
is a topological antilinear isomorphism.
\end{theorem}

\begin{proof}
Let 
$(F^p_{\phi})^{\bot}:=\{\,g\in L^q_{\phi}\,:\,\langle f, g\rangle_{\phi}=0, \mbox{ for every $f\in F^p_{\phi}$}\,\}$.
Then the fact that the map \eqref{dual:Lp} is an antilinear isometric isomorphism and a well-known consequence of the Hahn-Banach theorem show that the operator $S:L^q_{\phi}/(F^p_{\phi})^{\bot}\longrightarrow (F^p_{\phi})^{*}$, defined by
$S(g+(F^p_{\phi})^{\bot})= \langle\,\cdot\,,g\rangle_{\phi}$,
is an antilinear isometric isomorphism as well. On the other hand, 
$$
(F^p_{\phi})^{\bot}=\{\,g\in L^q_{\phi}\,:\,\langle f, P_{\phi}g\rangle_{\phi}=0, 
\mbox{ for every $f\in L^p_{\phi}$}\,\}=\{\,g\in L^q_{\phi}\,:\, P_{\phi}g=0\,\},
$$
where the first identity is a consequence of Corollary~\ref{corol:Pselfadjoint} and Theorem~\ref{thm:integrep}, while the second identity follows by duality. Therefore, since
$P_{\phi}$ is a bounded linear operator from the Banach space $L^q_{\phi}$ onto its closed subspace $F^q_{\phi}$, the open mapping theorem shows that the ``quotient'' operator
$\widetilde{P}_{\phi}:L^q_{\phi}/(F^p_{\phi})^{\bot}\longrightarrow F^q_{\phi}$, defined by 
$\widetilde{P}_{\phi}(g+(F^p_{\phi})^{\bot})=P_{\phi} g$, is a linear topological isomorphism.
Hence the operator  $S\circ(\widetilde{P}_{\phi})^{-1}$, which coincides with the mapping \eqref{dual:Fp},
 is an antilinear topological isomorphism, and the proof is complete.	
\end{proof}

\begin{corollary} \label{corol:dense}
	The linear span $E$ of all the reproducing kernels $K_z$, $z\in\C$, is dense in $\gFock$, for any $1\leq p< \infty$.
\end{corollary}

\begin{proof}
	By \thmref{thm:dual} and the Hahn-Banach theorem, we only have to prove that if $q$ is the conjugate exponent of $p$ and $f\in F^q_\phi$ satisfies $\psc{f, g}_\phi = 0$, for every $g\in E$, then $f=0$. And that follows from \thmref{thm:integrep}, since $f(z) = P_\phi f(z) = \psc{f, K_{z}}_\phi = 0$, for every $z\in\C$.
\end{proof}

\section{Fock-Carleson Measures}\label{sec:fock-carleson:measures}

Let $\M$ denote the set of all locally finite positive Borel measures on $\C$.
From now on it will be useful to consider the notion of {\em $p$-normalized reproducing kernel at $z\in \C$}:
$$
K_{p,z}(w) := \frac{K_z(w)}{\gfnorm{K_z}} \qquad (1\leq p <\infty, w\in \C).
$$ 

 The {\em Berezin transform of $\mu\in\M$} is defined to be
 $$
 \tld{\mu}(z) := \int_\C \abs{K_{2,z}(w)}^2 \efh{w} \d \mu(w) \qquad (z\in\C).
 $$
 For every $r>0$, the {\em $r$-averaging transform of $\mu\in\M$} is defined by
 $$
 \widehat{\mu}_r (z) := \frac{\mu(D^r(z))}{ A(D^r(z))}\simeq \frac{\mu(D^r(z))}{\rho(z)^2} \qquad (z\in \C).
 $$
\subsection{Fock-Carleson measures for $\gFock$}
Let $\mu\in\M$ and $1\leq p < \infty$. We say that $\mu$ is a \emph{Fock-Carleson measure} for $\gFock$ if there exists a constant $C>0$ such that
	\begin{equation} \label{eq:carleson1}
		\int_\C \abs{f \ee^{-\phi}}^p\d \mu \leq C \int_\C \abs{f \ee^{-\phi}}^p \d A, \qquad \textrm{for every\, } f\in\gFock.
	\end{equation}
In other words, $\mu$ is a Fock-Carleson measure for $\gFock$ when the inclusion operator
 $i_{p,\mu}\colon\gFock\hookrightarrow L^p(\C, \ee^{-p\phi}\d \mu)$ is bounded. Our next  result characterizes the Fock-Carleson measures for $\gFock$ in terms of the boundedness of their Berezin and averaging transforms.
 
\begin{theorem} \label{thm:carleson1}
Let $\mu\in\M$ and $1\le p< \infty$. Then the following assertions are equivalent:
	\begin{enumerate}
		\item $\mu$ is a Fock-Carleson measure for $\gFock$. \label{carl11}
		\item  There is $r>0$ such that $\widehat{\mu}_r \in L^{\infty}(\C)$. \label{carl12}
		\item $\tld{\mu} \in L^{\infty}(\C)$. \label{carl13}
	\end{enumerate}
	Moreover, 
$\|i_{p,\mu}\|^p\simeq\|\widehat{\mu}_r\|_{\infty}\simeq\|\tld{\mu}\|_{\infty}$.
\end{theorem}

An important consequence of this result is that the Fock-Carleson measures for $\gFock$ are independent of $p$, so we will simply call them \emph{$\phi$-Fock-Carleson measures}. In order to prove \thmref{thm:carleson1} we need the following lemmas.

\begin{lemma}\label{lemma:carleson1}
Let $0<p<\infty$ and assume $r_0>0$ satisfies \eqref{eq:kernel2}. 
Then for every $0< r\le r_0$ we have
$$
\widehat{\mu}_r (z) \simeq \int_{D^r(z)} \abs{K_{p,z}(w) \ef{w}}^p \d \mu(w) \qquad (\mu\in\M,\,z\in \C),
$$
and, in particular, 
$$
\widehat{\mu}_r (z) \lesssim \tld{\mu}(z)\qquad(\mu\in\M,\,z\in\C).
$$
\end{lemma}

\begin{proof}
	Note that \eqref{eq:rhoinballs}, \eqref{eq:kernel2} and \propref{prop:normkz} imply that
	$$\frac{1}{\rho(z)^2} \simeq \abs{K_{p,z}(w) \ef{w}}^p \qquad (z\in \C, w\in D^{r_0}(z)).$$
	Therefore for every $0<r \le r_0$ we have that
	$$\widehat{\mu}_r(z) \simeq \frac{\mu(D^r(z))}{\rho(z)^2} \simeq \int_{D^r(z)} \abs{K_{p,z}(w) \ef{w}}^p \d \mu(w) \qquad (z\in \C).$$
	The last assertion of the statement follows by taking $p=2$.
\end{proof}

\begin{lemma} \label{lemma:ball}
	Let $r>0$. Then there is $R>0$ such that $D^r(z) \subset \C\setminus D(0,\frac{\abs{z}}{2})$, for every $z\in\C$ with $\abs{z}\geq R$.
\end{lemma}

\begin{proof}
	If $w\in D^r(z)$ then $\abs{w} \geq \abs{z} - \abs{z-w} > \abs{z} - r \rho(z)$, and recall that, by \lemmaref{lemma:rhooutsideballs}, there is a constant $\delta\in (0,1)$ satisfying 
	$$\rho(z) \lesssim \abs{z}^{1-\delta} \qquad (z\in D^r(0)^c).$$
	Thus there is $R>0$ big enough such that $\abs{z} - r \rho(z) \geq \abs{z}/2$, if $\abs{z} \geq R$. Therefore it is clear that $R$ satisfies the statement of the lemma. 
\end{proof}

\begin{lemma}\label{lemma:mumean}
Let $0<p<\infty$. Then:
\begin{enumerate}[(a)]
\item For every $r > 0$ there is a constant $C>0$ such that \label{lemma:mumeana} 
$$
\int_\C \abs{f \ee^{-\phi}}^p \d \mu \leq C \int_\C \abs{f\ee^{-\phi}}^p \widehat{\mu}_r\d A, 
\,\,\mbox{ for any $\mu\in\M$ and $f\in\H(\C)$.}
$$
\item There is $R_0>0$ such that for every $r > 0$ there is a constant $C>0$ satisfying \label{lemma:mumeanb} 
$$
\int_{D(0,R)^c} \abs{f \ee^{-\phi}}^p \d \mu \leq C \int_{D(0,R/2)^c} \abs{f\ee^{-\phi}}^p \widehat{\mu}_r\d A,
$$
for any $\mu\in\M$, $R\ge R_0$ and $f\in \H(\C)$.
\end{enumerate}
\end{lemma}

\begin{proof}
By \lemmaref{lemma:estimates}, for every $s>0$ there is a constant $C_s>0$ such that
$$
\int_\Omega \abs{f(z) \ef{z}}^p \d \mu (z) \leq C_s \int_\Omega \int_{D^s(z)} \abs{f(w) \ef{w}}^p \frac{\d A(w)}{\rho(w)^2} \d\mu (z),
$$
	for any Borel set $\Omega \subset \C$ and for every $\mu\in\M$ and  $f\in\H(\C)$. Let $r>0$. Recall that, by \eqref{eq:rhoinballs}, there is a constant $c\geq 1$ such that $\rho(z) \leq c \rho(w)$, for every $z\in\C$ and $w\in D^r(z)$. Thus $z\in D^{sc}(w)$, whenever $0< s \leq r$ and $w\in D^s(z)$. Consequently, by Tonelli's theorem, we obtain that
	\begin{align*}
		\int_\Omega \abs{f(z) \ef{z}}^p \d \mu (z) & \leq C_s \int_{\Omega_r} \abs{f(w) \ef{w}}^p\frac{\mu(D^{sc}(w))}{\rho(w)^2} \d A (w),
	\end{align*}
	where $\Omega_r := \cup_{z\in\Omega} D^r(z)$, for any $s\in(0,r]$. By taking $s=s(r):=r/c \in (0,r]$ we have that
	\begin{equation}\label{eq:mumean}
		\int_\Omega \abs{f(z) \ef{z}}^p \d \mu (z) \leq C_{s(r)} \int_{\Omega_r} \abs{f(w) \ef{w}}^p\frac{\mu(D^{r}(w))}{\rho(w)^2} \d A (w).
	\end{equation}
	Since $\C_r = \C$, (\ref{lemma:mumeana}) directly follows from \eqref{eq:mumean}. On the other hand, by \lemmaref{lemma:ball}, there is $R_0>0$ such that $(D(0,R)^c)_r \subset D(0,R/2)^c$, for every $R\geq R_0$. Therefore (\ref{lemma:mumeanb}) also follows from \eqref{eq:mumean}.
\end{proof}

\begin{proof}[Proof of \thmref{thm:carleson1}]
	\imply{(\ref{carl11})}{(\ref{carl12})} By \lemmaref{lemma:carleson1} and applying \eqref{eq:carleson1} to $f=K_{p,z}$ we obtain that
		$$\widehat{\mu}_r(z) \simeq \int_{D^r(z)} \abs{K_{p,z}(w) \ef{w}}^p\d \mu(w) \leq C \gfnorm{K_{p,z}}^p = C \qquad (z\in\C).$$
	\imply{(\ref{carl12})}{(\ref{carl13})} By \lemmaref{lemma:mumean}(a) we have that
	$$
		\tld{\mu} (z)  \lesssim \int_\C \abs{K_{2,z} \ee^{-\phi}}^2 \widehat{\mu}_r \d A \leq \infnorm{ \widehat{\mu}_r} \quad (z\in\C).
	$$
	 \imply{(\ref{carl13})}{(\ref{carl11})} By Lemmas \ref{lemma:mumean}(a) and \ref{lemma:carleson1} we get that
	\begin{align*}
		\int_\C \abs{f \ee^{-\phi}}^p \d \mu & \lesssim \int_\C \abs{f \ee^{-\phi}}^p \widehat{\mu}_r \d A\lesssim \infnorm{ \tld{\mu}} \gfnorm{f}^p \quad (f\in \gFock).
	\end{align*}
	Moreover, we have proved the estimates
	$\|i_{p,\mu}\|^p\lesssim\|\tld{\mu}\|_{\infty}
	     \lesssim\|\widehat{\mu}_r\|_{\infty}\lesssim\|i_{p,\mu}\|^p$. 
\end{proof}

\subsection{Vanishing Fock-Carleson measures for $\gFock$}
Let $\mu\in\M$ and $1\leq p < \infty$. We say that $\mu$ is a \emph{vanishing Fock-Carleson measure for $\gFock$} when the inclusion operator $i_{p,\mu} \colon \gFock \hookrightarrow L^p(\C, \ee^{-p\phi}\d \mu)$ is compact.
Those measures are characterized by the fact that their Berezin and average transforms vanish at infinity, as the following result shows.
	
\begin{theorem}\label{thm:carleson2}
	Let $\mu\in\M$ and $1 \leq p < \infty$. Then the following statements are equivalent:
	\begin{enumerate}
		\item $\mu$ is a vanishing Fock-Carleson measure for $\gFock$. \label{carl21}
		\item There is $r>0$ such that $\widehat{\mu}_r(z) \rightarrow 0$, as $\abs{z}\to\infty$. \label{carl22}
		\item $\tld{\mu}(z) \rightarrow 0$, as $\abs{z}\to\infty$.\label{carl23}
	\end{enumerate}
\end{theorem}

An important consequence of this result is that the vanishing Fock-Carleson measures for $\gFock$ are independent of $p$, so we will simply call them \emph{vanishing $\phi$-Fock-Carleson measures}. The key tool in the proof of \thmref{thm:carleson2} is the following Kolmogorov-Riesz type compactness lemma.

\begin{lemma} \label{lemma:compact}
	Let $\nu\in\M$ and $1 \leq p <\infty$.
	\begin{enumerate}[(a)]
		\item If $\F$ is relatively compact in $L^p(\C, \dist\nu)$ then $\F$ is bounded in $L^p(\C, \dist\nu)$ and satisfies 
		\begin{equation} \label{eq:cues}
			\lim_{R\to\infty} \sup_{f\in\F} \norm{f \chi_{D(0,R)^c}}_{L^p(\C, \dist\nu)} = 0.
		\end{equation}
		\item If $\F$ is a locally bounded family of entire functions satisfying \eqref{eq:cues} then $\F$ is relatively compact in $L^p(\C, \dist\nu)$. 
	\end{enumerate}
\end{lemma}

\begin{proof}
	(a) Assume that $\F$ is relatively compact in $L^p(\C, \dist\nu)$. Then it is clear that $\F$ is bounded in $L^p(\C, \dist\nu)$, so we are going to show that it satisfies \eqref{eq:cues}. 
	
	Let
	$$L:= \varlimsup_{R\to\infty} \sup_{f\in\F} \norm{f \chi_{D(0,R)^c}}_{L^p(\C, \dist\nu)},$$
	and note that \eqref{eq:cues} is equivalent to $L=0$. In order to prove that, pick a sequence of functions $(f_n)_n$ in $\F$ and a sequence of positive numbers $(R_n)_n$ such that $R_n \to\infty$ and $\norm{f_n \chi_{D(0,R_n)^c}}_{L^p(\C, \dist\nu)} \longrightarrow L$, as $n\to\infty$. Since $\F$ is relatively compact in $L^p(\C, \dist\nu)$, there is a subsequence $(f_{n_k})_k$ of $(f_n)_n$ which converges in $L^p(\C, \dist\nu)$ to a function $f\in L^p(\C, \dist\nu)$. Then 
	\begin{align*}
		\norm{f_{n_k} \chi_{D(0,R_{n_k})^c}}_{L^p(\C, \dist\nu)} & \leq \norm{(f_{n_k} - f)\chi_{D(0,R_{n_k})^c}}_{L^p(\C, \dist\nu)} + \norm{f\chi_{D(0,R_{n_k})^c} }_{L^p(\C, \dist\nu)}\\
		& \leq \norm{f_{n_k} - f}_{L^p(\C, \dist\nu)} + \norm{f\chi_{D(0,R_{n_k})^c} }_{L^p(\C, \dist\nu)},
	\end{align*}
	and letting $k\to \infty$ we get that $L=0$.
	
	(b) Let $\F$ be a locally bounded family of entire functions which satisfies \eqref{eq:cues}. Since $\nu$ is locally finite, for every $R>0$, we have that
	\begin{equation}\label{eq:locfinite}
		\norm{f \chi_{D(0,R)}}_{L^p(\C, \dist\nu)} \lesssim \sup_{\abs{z} \leq R} \abs{f(z)} \qquad (f\in \Cont(\cl{D(0,R)}),
	\end{equation}
	where as usual $\Cont(\cl{D(0,R)})$ is the space of continuous functions on the closed disk $\cl{D(0,R)}$. Then it is clear that \eqref{eq:cues} and \eqref{eq:locfinite} show that $\F\subset  L^p(\C, \dist\nu)$. Now we want to prove that $\F$ is relatively compact in $L^p(\C, \dist\nu)$, or equivalently that $\F$ is precompact (totally bounded) in $L^p(\C, \dist\nu)$, which means that for every $\e >0$ there is a finite covering of $\F$ by balls in $L^p(\C, \dist\nu)$ of radius $\e$. 
	
	Let $\e>0$. By \eqref{eq:cues} there is $R>0$ such that
	$$\sup_{f\in\F} \norm{f \chi_{D(0,R)^c}}_{L^p(\C, \dist\nu)} < \e/4.$$
	Since $\F$ is locally bounded,
	$$\sup\{ \abs{f(z)} \colon f\in\F, \abs{z} \leq 2R \} < \infty$$
	and so $\F$ is a normal family on the disk $D(0,2R)$, by Montel's theorem. In particular, $\F$ is relatively compact (and so precompact) in $\Cont(\cl{D(0,R)})$. Taking into account \eqref{eq:locfinite}, it follows that there are finitely many functions $f_1, \ldots, f_n$ in $\F$ such that for any $f\in \F$ there is $1\leq j \leq n$ so that $\norm{(f-f_j) \chi_{D(0,R)}}_{L^p(\C, \dist\nu)} < \e/2$ and therefore
	\begin{align*}
		\norm{f-f_j}_{L^p(\C, \dist\nu)} & \leq  \norm{(f-f_j) \chi_{D(0,R)}}_{L^p(\C, \dist\nu)} + \norm{(f-f_j) \chi_{D(0,R)^c}}_{L^p(\C, \dist\nu)}  \\
		& < \frac{\e}{2} +\norm{f \chi_{D(0,R)^c}}_{L^p(\C, \dist\nu)} + \norm{f_j \chi_{D(0,R)^c}}_{L^p(\C, \dist\nu)} < \e.
	\end{align*}
	Hence the proof is complete.
\end{proof}

\begin{corollary}\label{corol:compact}
	Let $1\leq p <\infty$ and $\F \subset \gFock$. Then $\F$ is relatively compact in $\gFock$ if and only if $\F$ is bounded in $\gFock$ and satisfies
	\begin{equation} \label{eq:compact}
		\lim_{R\to\infty} \sup_{f\in\F} \norm{f \chi_{D(0,R)^c}}_{p,\phi} = 0.
	\end{equation}
\end{corollary}

\begin{proof}
	It follows by applying \lemmaref{lemma:compact} to $\dist \nu = \ee^{-p\phi} \dist A$. Namely, part (a) proves that if $\F$ is relatively compact in $\gFock$ then $\F$ is bounded in $\gFock$ and satisfies \eqref{eq:compact}. Recall that \eqref{eq:est} shows that if $\F$ is bounded in $\gFock$ then $\F$ is locally bounded, and so part (b) completes the proof of the corollary.
\end{proof}

\begin{lemma} \label{lemma:kzweak}
	 $\displaystyle \lim_{\abs{z} \to\infty} K_{2,z}(w) = 0$, for every $w\in\C$.
\end{lemma}

\begin{proof}
	Let $w\in\C$. Then, by \eqref{prop:kerneldiagonal}, \lemmaref{lemma:estimates} and \eqref{eq:rhoinballs}, we have that
	\begin{align*}
		\abs{K_{2,z}(w)} & = \abs{\fpsc{K_{2,z}, K_w}} \simeq \rho(z)\abs{K_w(z)} \ef{z} \lesssim \norm{K_w \chi_{D(z)}}_{2,\phi}.
	\end{align*}
	So, by \lemmaref{lemma:ball}, $\abs{K_{2,z}(w)}  \lesssim \norm{K_w \chi_{D(0,\abs{z}/2)^c}}_{2,\phi}\to 0,$
	as $\abs{z}\to \infty$, and we are done.
\end{proof}

%
%

\begin{proof}[Proof of \thmref{thm:carleson2}]
	\imply{(\ref{carl21})}{(\ref{carl22})} The hypothesis shows that $\F:= \{f\in\gFock \, \colon \norm{f}_{p,\phi} \leq 1\}$ is relatively compact in $L^p(\C,\dist\nu)$, where $\dist \nu = \ee^{-p\phi}\dist \mu$. So \lemmaref{lemma:compact}(a) implies that \eqref{eq:cues} holds. Then, by Lemmas~\ref{lemma:carleson1} and \ref{lemma:ball}, there is $r>0$ such that
	$$\widehat{\mu}_r(z)\simeq \norm{K_{p,z}\chi_{D^r(z)}}_{L^p(\C,\dist\nu)}^p \leq  \norm{K_{p,z}\chi_{D(0,\abs{z}/2)^c}}_{L^p(\C,\dist\nu)}^p  \longrightarrow 0, \quad \textrm{as\ } \abs{z}\to\infty.$$
	\imply{(\ref{carl22})}{(\ref{carl23})} By \lemmaref{lemma:mumean}(a) there is a constant $C>0$ such that
	\begin{align*}
		\tld{\mu}(z) = \int_\C \abs{K_{2,z} e^{-\phi}}^2 \d \mu& \leq C \int_\C \abs{K_{2,z}e^{-\phi}}^2  \widehat{\mu}_r \d A.
	\end{align*}
	Then, for every $s>0$, we have that 
	$$\tld{\mu}(z) \leq C \infnorm{ \widehat{\mu_r}\chi_{D^s(0)}} \int_{D^s(0)} \abs{K_{2,z} e^{-\phi}}^2 \d A + C\sup_{w\in D^s(0)^c}\widehat{\mu}_r(w).$$
	The limit of the second term of the above sum is $0$, as $s\to \infty$, by the hypothesis. The locally finiteness of $\mu$ and \eqref{eq:rhoinballs} show that $ \infnorm{ \widehat{\mu}_r\chi_{D^s(0)}} < \infty$, for every $s>0$. Then \lemmaref{lemma:kzweak},  \eqref{eq:est}, \eqref{eq:rhoinballs} and the dominated convergence theorem imply that the first term goes to $0$ as $\abs{z} \to \infty$, for every $s>0$. Hence $\tld{\mu}(z) \to 0$, as $\abs{z} \to \infty$.
	 \imply{(\ref{carl13})}{(\ref{carl11})} We are going to prove that $\mu$ is a vanishing Fock-Carleson measure for $\gFock$
        by applying \lemmaref{lemma:compact}(b) to $\dist \nu = \ee^{-p\phi} \dist \mu$ and $\F:= \{f\in \gFock \,\colon \norm{f}_{p,\phi} \leq 1\}$. By \eqref{eq:est} and \eqref{eq:rhoinballs} it is clear that $\F$ is a locally bounded family of entire functions. Now we want to prove that \eqref{eq:cues} holds. By \lemmaref{lemma:mumean}(b) there are constants $r, R_0, C >0 $ satisfying 
	 $$\int_{D(0,R)^c} \abs{f \ee^{-\phi}}^p \d \mu \leq C \int_{D(0,R/2)^c} \abs{f\ee^{-\phi}}^p \widehat{\mu}_r\d A,$$
	for any $R \geq R_0$ and $f\in \H(\C)$. Therefore by \lemmaref{lemma:carleson1}
$$
\sup_{f\in\F} \int_{D(0,R)^c} \abs{f \ee^{-\phi}}^p \d \mu \lesssim \sup_{\abs{z}\geq R/2} \tld{\mu} (z)
\qquad(R\geq R_0),
$$
 and hence the hypothesis $\tld{\mu}(z) \rightarrow 0$, as $\abs{z} \to \infty$, implies that \eqref{eq:cues} holds.
\end{proof}

\section{Toeplitz Operators} \label{sec:toeplitz}

The {\em Toeplitz operator $\Tz$ with symbol $\mu\in\M$} is defined to be
$$
T_\mu f(z) := \int_\C f(w)\overline{K_z(w)} \efh{w} \d\mu (w) \quad (z\in \C).
$$
Note that $T_\mu f$ is defined if the function $f$ satisfies $f\,\overline{K}_z\, e^{-2\phi}\in L^1(\C,\d\mu)$, for every $z\in\C$.  

The goal of this section is to study the boundedness and compactness of  the Toeplitz operator $\Tz$ on $\gFock$  in terms of the symbol $\mu$.

\subsection{Boundedness}

In this subsection we characterize the boundedness of the Toeplitz operator $\Tz$ acting on $\gFock$ for $1\leq p <\infty$. Recall that  we say that $\Tz$ is bounded on $\gFock$ when, for every $f\in\gFock$,  $\Tz f$ is an entire function and  $\|\Tz f\|_{p,\phi}\lesssim\|f\|_{p,\phi}$.  

\begin{theorem}\label{thm:Tzbounded}
	Let $\mu\in\M$ and $1 \leq p < \infty$. Then the following assertions are equivalent:
	\begin{enumerate}
\item $T_{\mu}$ is bounded on $\gFock$.  \label{bound:first}
\item $\Tz K_{p,z}\in\gFock$, for every $z\in\C$, and $M_{p,\mu}:=\sup_{z\in\C}\|T_{\mu}K_{p,z}\|_{p,\phi}<\infty$.\label{bound:second}
\item $\mu$ is a $\phi$-Fock-Carleson measure.\label{bound:third}
	\end{enumerate}
Moreover, 
 $\norm{\Tz}_{\gFock\to\gFock}\simeq M_{p,\mu}\simeq\|i_{p,\mu}\|^p$.	
\end{theorem}
The fact that the linear span of the reproducing kernels $K_{p,z}$ is dense in $\gFock$ (see~\corolref{corol:dense})  justifies somehow the occurrence of the assertion (2) in the statement of \thmref{thm:Tzbounded}.

To prove \thmref{thm:Tzbounded} we need the following two lemmas.

\begin{lemma} \label{lemma:Tzbnd}
	 Let $\mu\in\M$ and $1 \leq p < \infty$. Assume that $\Tz K_z \in \gFock$, for every $z\in\C$. Then, for every $r>0$, we have that
	$$
	\tld{\mu}(z) \lesssim \norm{\chi_{D^r(z)}\Tz K_{p,z} }_{p,\phi}
	 \qquad (z\in\C).
	$$
\end{lemma}

\begin{proof}
	It follows from Proposition~\ref{prop:normkz}, \lemmaref{lemma:estimates} and \eqref{eq:rhoinballs}:
$$
\tld{\mu}(z)
 = \frac{\Tz K_z(z)}{\norm{K_z}_{2,\phi}^2} \simeq  \frac{\Tz K_{p,z}(z)}{\norm{K_z}_{q,\phi}}
 \lesssim \frac{\ee^{\phi(z)}\norm{\chi_{D^r(z)}\Tz K_{p,z}}_{p,\phi}}{ \norm{K_z}_{q,\phi} \rho(z)^{2/p}} \simeq \norm{\chi_{D^r(z)}\Tz K_{p,z}}_{p,\phi},
$$
where $q$ is the conjugate exponent of $p$.
\end{proof}

\begin{lemma} \label{fock-carleson-measure:bounded}\hspace*{\fill}
\begin{enumerate}[(a)]
\vspace*{3pt}
\item If $\mu\in\M$ then
\begin{equation}\label{definition:modulus:Toeplitz:operator}
\TT f(z):=\int_{\C} f(w)|K_z(w)|e^{-2\phi(w)}\d\mu(w)		
\end{equation}
defines a bounded linear operator from $L^1(\C,e^{-\phi}\d\mu)$
to $L^1_{\phi}$, and its norm $\|\TT\|_1$ satisfies 
$\|\TT\|_1\lesssim 1$.
\vspace*{3pt}
\item If $\mu$ is a $\phi$-Fock-Carleson measure then \eqref{definition:modulus:Toeplitz:operator} defines a bounded linear operator from $L^p(\C,e^{-p\phi}\d\mu)$
to $L^p_{\phi}$, for every $1<p<\infty$, and its norm 
$\|\TT\|_p$ satisfies $\|\TT\|_p\lesssim \|i_{p,\mu}\|^{p-1}$.  
\end{enumerate}	
\end{lemma}

\begin{proof}

\noindent
(a) It follows from Tonelli's theorem and
 \lemmaref{lemma:kernel}(b):
\begin{align*}
 \|\TT f\|_{1,\phi}
 &\le \int_{\C}\left(
  \int_{\C}|f(w)||K_z(w)|e^{-2\phi(w)}\d\mu(w)
  \right)e^{-\phi(z)}\d A(z) \\
 &=  \int_{\C}\left(
 \int_{\C}|K_w(z)|e^{-\phi(z)}\d A(z)
 \right)|f(w)|e^{-2\phi(w)}\d\mu(w) \\
 &\lesssim \int_{\C}|f(w)|e^{-\phi(w)}\d\mu(w)=
 \|f\|_{L^1(\C,e^{-\phi}\d\mu)}.
\end{align*}
	
\noindent
(b) Assume $\mu$ is a $\phi$-Fock-Carleson measure. First, H\"older's inequality shows that
$$
|\TT f(z)e^{-\phi(z)}|\le 
\left(\int_{\C}
|f(w)e^{-\phi(w)}|^p|K_z(w)|e^{-\phi(w)-\phi(z)}\d\mu(w)
\right)^{1/p} M(z)^{1/q},
$$	
where $q$ is the conjugate exponent of $p$ and 
$$
M(z):=\int_{\C}|K_z(w)|e^{-\phi(w)-\phi(z)}\d\mu(w)
\le\|i_{1,\mu}\|\,\|K_z\|_{1,\phi} 
\,e^{-\phi(z)}\lesssim\|i_{1,\mu}\|,
$$
by the hypothesis and~{\lemmaref{lemma:kernel}(b).}
 Therefore
 $$
 |\TT f(z) e^{-\phi(z)}|^p\lesssim \|i_{1,\mu}\|^{p/q}
\int_{\C}|f(w)e^{-\phi(w)}|^p|K_z(w)|e^{-\phi(z)-\phi(w)}
\d\mu(w),
 $$
 and hence Tonelli's theorem, \lemmaref{lemma:kernel}(b)
  and~{\thmref{thm:carleson1}}  imply that
\begin{align*}
\|\TT f\|^p_{p,\phi}
&\lesssim \|i_{1,\mu}\|^{p/q} \int_{\C} 
\left(\int_{\C}|K_w(z)|e^{-\phi(z)-\phi(w)}\d A(z)\right)
|f(w)e^{-\phi(w)}|^p\d\mu(w) \\
&\lesssim \|i_{1,\mu}\|^{p/q}\|f\|^p_{L^p(\C,e^{-p\phi}\d\mu)}
\simeq \|i_{p,\mu}\|^{p(p-1)}\|f\|^p_{L^p(\C,e^{-p\phi}\d\mu)}.
\qedhere
\end{align*}  
\end{proof}

\begin{proof}[Proof of \thmref{thm:Tzbounded}]
\imply{(\ref{bound:first})}{(\ref{bound:second})} 
The boundedness of $\Tz$ on $\gFock$ clearly implies that 
$\Tz K_{p,z}\in\gFock$, for every $z\in\C$, and 
$M_{p,\mu}\le \norm{\Tz}_{\gFock\to\gFock}$.
\imply{(\ref{bound:second})}{(\ref{bound:third})} 
The hypotheses and \lemmaref{lemma:Tzbnd} show that
$$
\tld{\mu}(z) \lesssim \gfnorm{\Tz K_{p,z}} 
\le M_{p,\mu} \qquad (z\in\C), 
$$
so, by \thmref{thm:carleson1}, $\mu$ is a 
$\phi$-Fock-Carleson measure and 
$\|i_{p,\mu}\|^p\lesssim M_{p,\mu}$.
\imply{(\ref{bound:third})}{(\ref{bound:first})} Assume that 
$\mu$ is a $\phi$-Fock-Carleson measure. First we want to prove that
$\Tz f\in\H(\C)$, for every $f\in\gFock$. We proceed by 
differentiation under the integral sign as in the proof 
of~{\thmref{thm:projbnd}}. Following that proof we only have to
 check that $G_{z_0}\in L^1(\C,\d\mu)$, which is proved by 
 using the same arguments which show that 
 $G_{z_0}\in L^1(\C)$, since $\mu$ is a $\phi$-Fock-Carleson 
 measure.     

Now the boundedness of $\Tz$ on $\gFock$ follows from 
\lemmaref{fock-carleson-measure:bounded} and our hypothesis:
$$
\gfnorm{\Tz f}\le\|\TT f\|_{p,\phi}
\lesssim\|i_{p,\mu}\|^{p-1}\|f\|_{L^p(\C, e^{-p\phi}\d\mu)}
\le\|i_{p,\mu}\|^p\|f\|_{p,\phi}\quad(f\in F^p_{\phi}).
$$
In particular,  $\norm{\Tz}_{\gFock\to\gFock}\lesssim\|i_{p,\mu}\|^p$.
\end{proof}

\subsection{Compactness}

In this section we characterize the compactness of the Toeplitz operator $\Tz$ on $\gFock$ for $1\leq p < \infty$. 

\begin{theorem} \label{thm:Tzcompact}
	Let $\mu\in\M$ and $1\leq p < \infty$. Then the following statements are equivalent:
	\begin{enumerate}
		\item $\Tz$ is compact on $\gFock$. \label{compact:first}
		\item $\mu$ is a vanishing $\phi$-Fock-Carleson measure. \label{compact:second}
	\end{enumerate}
\end{theorem}

 In order to prove \thmref{thm:Tzcompact} we need the following lemma.

\begin{lemma} \label{lemma:Tzcompact}
For $R,S>0$ let
\begin{equation} \label{M(R,S)}
M(R,S):=\sup_{z\in D(0,S)}e^{-\phi(z)}\int_{D(0,R)^c}|K_z(w)|e^{-\phi(w)}\d A(w).
\end{equation}
Then $M(R,S)\to 0$, as $R\to\infty$, for every $S>0$.
\end{lemma}

\begin{proof}
Let $S>0$ and $\delta_S:=\sup_{z\in D(0,S)}\rho(z)$. Then 
$D^r(z)\subset D(0,S+r\delta_S)$, for every $r>0$ and $z\in D(0,S)$.
Therefore $r(R,S):=(R-S)/\delta_S$ satisfies that 
$$
D(0,R)^c\subset D^{r(R,S)}(z)^c,\quad\mbox{for any  $R>S$ and $z\in D(0,S)$.}
$$
Moreover, note that  $r(R,S)\to \infty$, as $R\to\infty$.
Hence, by \lemmaref{lemma:kernel}(a), we conclude that $M(R,S)\le C(r(R,S))\to0$, as $R\to\infty$, and we are done.
\end{proof}

\begin{proof}[Proof of \thmref{thm:Tzcompact}] 
	Let $\F:= \{f\in\gFock : \norm{f}_{p,\phi}\leq 1\}$ be the closed unit ball in $\gFock$.
	\imply{(\ref{compact:first})}{(\ref{compact:second})} The compactness of $\Tz$ on $\gFock$ and \corolref{corol:compact} show that $\Tz \F$ satisfies \eqref{eq:compact}. Then Lemmas~\ref{lemma:Tzbnd} and \ref{lemma:ball} imply that
$$
	\tld{\mu}(z)  \lesssim \norm{(\Tz K_{p,z}) \chi_{D^r(z)}}_{p,\phi}
                           \le\norm{(\Tz K_{p,z})\chi_{D(0,\abs{z}/2)^c}}_{p,\phi} \rightarrow 0,
$$
	as $\abs{z}\to\infty$. Therefore (\ref{compact:second}) holds by \thmref{thm:carleson2}.
	\imply{(\ref{compact:second})}{(\ref{compact:first})} Assume that $\mu$ is a vanishing $\phi$-Fock-Carleson measure. Then $\mu$ is also a $\phi$-Fock-Carleson measure and so, by \thmref{thm:Tzbounded}, $\Tz$ is bounded on $\gFock$. 
	
	Now we are going to show that $\Tz$ is compact on $\gFock$, that is, $\Tz \F$ is relatively compact in $\gFock$. In order to do that we will apply \corolref{corol:compact}. Since $\Tz$ is bounded on $\gFock$, $\Tz\F$ is bounded in $\gFock$, so it only remains to prove that $\Tz\F$ satisfies \eqref{eq:compact}. 

	Let $f\in\gFock$ and $R>0$. Then
\begin{equation}\label{trivial:estimate:cuep:Tmu}
\gfnorm{(\Tz f) \chi_{D(0,R)^c}}^p\le
\int_{D(0,R)^c}I(z)^pe^{-p\phi(z)}\d A(z),
\end{equation}
where
$$
I(z):=\int_{\C}|f(w)||K_z(w)|e^{-2\phi(w)}\d\mu(w).
$$
Now, if $p>1$ and $q$ is the conjugate exponent of $p$, $I(z)$ is estimated by using H\"{o}lder's inequality, the  fact that $\mu$ is a $\phi$-Fock-Carleson measure and \lemmaref{lemma:kernel}(b) as follows
\begin{eqnarray*}
I(z)
&\le&  \left(\int_{\C}|f(w)|^p|K_z(w)|e^{-(p+1)\phi(w)}\d\mu(w)\right)^{\frac1p}
          \left(\int_{\C}|K_z(w)|e^{-\phi(w)}\d\mu(w)\right)^{\frac1q} \\
&\lesssim&  e^{\frac{\phi(z)}q}   \left(\int_{\C}|f(w)|^p|K_z(w)|e^{-(p+1)\phi(w)}\d\mu(w)\right)^{\frac1p}.
\end{eqnarray*}
So, for $p>1$, we get
$$
\gfnorm{(\Tz f) \chi_{D(0,R)^c}}^p
\lesssim \int_{D(0,R)^c} \left(\int_{\C}|f(w)|^p|K_w(z)|e^{-(p+1)\phi(w)}\d\mu(w)\right) e^{-\phi(z)}\d A(z).
$$
Note that \eqref{trivial:estimate:cuep:Tmu} shows that this estimate also holds for $p=1$.
Therefore, by Tonelli's theorem, we have that
 $$
\gfnorm{(\Tz f) \chi_{D(0,R)^c}}^p
\lesssim \int_{\C}|f(w)|^pe^{-(p+1)\phi(w)}\left(\int_{D(0,R)^c} |K_z(w)|e^{-\phi(z)}\d A(z)\right) \d\mu(w).
$$
For every $S>0$, we split the above integral on $\C$ into the corresponding integrals on $D(0,S)$ and $D(0,S)^c$, which we denote by $I_S(R)$ and $J_S(R)$, respectively. 
Then \lemmaref{lemma:kernel}(b) implies that 
$$
J_S(R)\lesssim \|f\chi_{D(0,S)^c}\|^p_{L^p(\C,e^{-p\phi}\d\mu)}.
$$
Moreover, since $\mu$ is a $\phi$-Fock-Carleson measure, 
$I_S(R)\lesssim M(R,S) \|f\|_{p,\phi}$, where $M(R,S)$ is defined by \eqref{M(R,S)}. 

Therefore it turns out that there is a constant $C>0$ such that
 \begin{equation} \label{eq:tzcompact}
\sup_{f\in\F} \gfnorm{(\Tz f) \chi_{D(0,R)^c}}^p 
\leq C\left\{M(R,S)+  \sup_{f\in\F} \norm{f  \chi_{D(0,S)^c} }^p_{L^p(\C, \ee^{-p\phi}\dist\mu)}\right\},
\end{equation}
for every $R,S>0$.
	Since $\mu$ is a vanishing $\phi$-Fock-Carleson measure for $\gFock$, $\F$ is relatively compact in $L^p(\C, \ee^{-p\phi}\dist\mu)$
 and \lemmaref{lemma:compact}(a) shows that the second summand of the right-hand side term of \eqref{eq:tzcompact} goes to $0$, as $S\to \infty$. 
Moreover, $M(R,S)\to0$, as $R\to\infty$, for every $S>0$, by \lemmaref{lemma:Tzcompact}.
Hence we conclude that $\Tz\F$ satisfies~{\eqref{eq:compact}.}
\end{proof}
	
\section{Membership in the Schatten Class of $\gFockh$}
\label{sec:membership:shatten:class}

  Let $H$ be a separable complex Hilbert space.
 Recall that if $T$ is a positive operator on $H$ and $(e_n)_n$ is an orthonormal basis of $H$, then the quantity
 $$
 \sum_n\langle Te_n,e_n\rangle\in [0,\infty]
 $$
 does not depend on the basis $(e_n)_n$. It is called the {\em trace} of $T$ and it is denoted by $\tr{T}$. It is well known that if $\tr{T}<\infty$ then $T$ is compact and $\tr{T}=\sum_n\lambda_n$, where $(\lambda_n)_n$ is the sequence of eigenvalues of $T$.
 
 For $0<p<\infty$, the Schatten class $\S_p=\S_p(H)$ of $H$ is the set of all bounded linear operators $T$ on $H$ such that $\norm{T}_{\S_p}^p :=\tr{|T|^p}<\infty$,
 where $|T|$ is the positive operator defined by $|T|:= (T^* T)^{\frac{1}{2}}$.
 Then it is clear that any $T\in\S_p(H)$ is a compact operator and  $\norm{T}_{\S_p}^p= \sum_n\lambda_n^p$, where $(\lambda_n)_n$ is the sequence of eigenvalues of $|T|$.
Moreover, $(\S_p(H),\|\cdot\|_{\S_p})$ is a Banach space for $1\le p<\infty$, and a quasi-Banach space for $0<p<1$.
Furthermore, we have the triangular inequality 
\begin{equation}\label{rotfel'd:inequality}
\|T+S\|^p_{\S_p}\le\|T\|^p_{\S_p}+\|S\|^p_{\S_p},\qquad
(S,T\in\S_p(H), 0<p<1),
\end{equation} 
which is called {\em Rotfel'd inequality} (see \cite{rotfeld1,rotfeld2,thompson}).

 We refer to \cite[Chapter 1]{zhu2007operator} for the basic properties of the Schatten class operators.

In this section we want to study when $\Tz\in\S_p(\gFockh)$.
Note that if $\mu$ is a $\phi$-Fock-Carleson measure then Fubini's theorem and \thmref{thm:integrep} show that
\begin{equation}\label{positiveness:Toeplitz:operator}
\langle\Tz f, f\rangle_{\phi} = 
\int_{\C}|f|^2 e^{-2\phi} \d\mu,  
\quad \mbox{for every $f\in F^2_{\phi}$.}
\end{equation}
(The hypothesis of Fubini's theorem is fulfilled due to
Cauchy-Schwarz inequality, the boundedness of the operator 
 $\TT$ on $F^2_{\phi}$ 
 (by \lemmaref{definition:modulus:Toeplitz:operator}) and our 
 assumption  that $\mu$ is a $\phi$-Fock-Carleson measure). 
 Thus if $\Tz$ is bounded on $\gFockh$ then 
 \eqref{positiveness:Toeplitz:operator}  holds, by 
 \thmref{thm:Tzbounded}, so
 $\Tz$ is a positive operator on $\gFockh$ and, in particular, 
 $\abs{\Tz} = \Tz$.
 Therefore $\Tz\in\S_p(\gFockh)$ if and only if $\Tz$ is bounded on $\gFockh$
 and $\tr{T_{\mu}^{p}} < \infty$.

In order to state the characterization of the membership of 
$\Tz$ in $\S_p(\gFockh)$ we need the concept
of $(r,\phi)$-lattice.
For any $r>0$, an {\em $(r,\phi)$-lattice} is
 a sequence of different points in $\C$ such that $\{D^r(z_j)\}_{j\ge1}$ is a covering of $\C$ satisfying
 \begin{equation}\label{overlapping:index}
 N_r(\{z_j\}_{j\ge1}):=
 \sup_{z\in\C}\sum_{j=1}^{\infty}\chi_{D^r(z_j)}(z)<\infty.
 \end{equation}
The existence of $(r,\phi)$-lattices, for any $r>0$, is 
guaranteed by \cite[Proposition 7]{dall}.

We characterize the membership of $\Tz$  in the Schatten class 
$\S_p(\gFockh)$ as follows.

\begin{theorem} \label{thm:Tzsp}
Let $\mu\in\M$ and $0< p < \infty$. Then the following 
statements are equivalent:
\begin{enumerate}
\item \label{sp:first} $\Tz\in\S_p(\gFockh)$.
\item \label{sp:second}
       There is $r_0>0$ such that any $(r,\phi)$-lattice $\{z_j\}_{j\ge1}$ with $r\in(0,r_0)$ satisfies $\{\widehat{\mu}_r(z_j)\}_{j\ge1}\in\ell^p$.
\item \label{sp:third}
       There is an $(r,\phi)$-lattice $\{z_j\}_{j\ge1}$ such that 
       $\{\widehat{\mu}_r(z_j)\}_{j\ge1}\in\ell^p$. 
\item \label{sp:four} There is $r>0$ such that    
      $\widehat{\mu}_r \in L^p(\C,\!\d \sigma)$.                
\item \label{sp:fith} $\tld{\mu} \in L^p(\C,\! \d \sigma)$. 
\end{enumerate}
	Moreover, $\norm{T_{\mu}}_{\S_p}^p\simeq
	\|\widehat{\mu}_r\|_{L^p(\C,\!\d \sigma)}\simeq
	\|\tld{\mu}\|_{L^p(\C,\!\d \sigma)}$. 
	Recall that $\dist \sigma = \dist A/\rho^2$.
\end{theorem}

We are going to prove \thmref{thm:Tzsp} in the next  subsections.

\subsection{Some technical lemmas}
In this subsection we collect all the technical lemmas that we need to carry out the proof  of \thmref{thm:Tzsp}. 

\begin{lemma} \label{lemma:kappa:constant}
	Let $0<r<1/2$ and $z\in\C$. Then:
	\begin{enumerate}
		\item[(a)] $r^2/2\le\sigma(D^r(z))\le 4^2r^2$.\vspace*{4pt}
		\item[(b)] 
		$\widehat{\mu}_{r/4}(z)\le 4\,\widehat{\mu}_r(w)$
		and
		$\widehat{\mu}_{r/4}(w)\le 4\,\widehat{\mu}_r(z)$, for any $\mu\in\M$ and  $w\in D^{r/4}(z)$.
	\end{enumerate}
\end{lemma}

\begin{proof}
	By \lemmaref{lemma:rhoinballs}, 
	$\rho(w)/2\le\rho(z)\le 2\rho(w)$, for every $z\in\C$ and 
	$w\in D^{1/2}(z)$.
	Then $\pi r^2/4\le\sigma(D^r(z))\le4\pi r^2$, for any $r\in(0,1/2)$ and $z\in\C$, so  (a) holds. Moreover,  it is easy to check that 
	$D^{r/4}(z)\subset D^r(w)$ and 
	$D^{r/4}(w)\subset D^r(z)$, 
	for any $r\in(0,1/2)$, $z\in\C$ and $w\in D^{r/4}(z)$.
	Therefore (b) directly follows.
\end{proof}

\begin{lemma} \label{lemma:trace}
	Let $T$ be a positive operator on $\gFockh$. Then the trace of $T$ is
	$$
	\tr{T}=\int_{\C}\tld{T}(z)\d\sigma(z),
	$$
	where $\tld{T}(z):=\psc{T K_{2,z}, K_{2,z}}_\phi$ is the Berezin transform of $T$.
\end{lemma}

\lemmaref{lemma:trace} is proved as \cite[Proposition 3.3]{zhu2012fock}, so we omit the proof.
\begin{lemma}[{\cite[Proposition 1.31]{zhu2007operator}}] \label{lemma:p-powers:of:operators}
	Let $T$ be a positive operator on a complex Hilbert space and let $x$ be a unit
	vector in $H$. Then:
	
	\begin{enumerate}
		\item[(a)] $\langle T^px,x\rangle\ge\langle Tx,x\rangle^p$,
		for every $1\le p<\infty$.
		\item[(b)] $\langle T^px,x\rangle\le\langle Tx,x\rangle^p$,
		for every  $0<p\le1$.
	\end{enumerate}
\end{lemma}

\begin{lemma} \label{lemma:aver:in:Lp:Tmu:in:Sp}
	Assume that $\mu$ is a $\phi$-Fock-Carleson measure. Then:
	
	\begin{enumerate}
		\item[(a)] $\norm{\tld{\mu}}^p_{L^p(\C,\d\sigma)}\le\tr{\Tz^p}$, for every  $1\le p<\infty$.
		\item[(b)] $\tr{\Tz^p}\le\norm{\tld{\mu}}^p_{L^p(\C,\d\sigma)}$, for every $0<p\le1$.
	\end{enumerate}
\end{lemma}

\begin{proof}
	Let $0<p<\infty$. Then $T_{\mu}^p$ is a positive operator on $\gFockh$, and so,
	by \lemmaref{lemma:trace}, 
	$$
	\tr{\Tz^p}=\int_{\C}\tld{\Tz^p}(z)\d\sigma(z).
	$$
	Now \lemmaref{lemma:p-powers:of:operators} shows that:
	\begin{enumerate}
		\item[(a)]  If $1\le p<\infty$ then 
		$\tld{\Tz^p}(z)\ge (\tld{\Tz}(z))^p$, for every $z\in\C$.
		\item[(b)]  If $0<p\le1$ then $\tld{\Tz^p}(z)\le (\tld{\Tz}(z))^p$, for every $z\in\C$.
	\end{enumerate}
	By \eqref{positiveness:Toeplitz:operator}, $\tld{\Tz}=\tld{\mu}$, and hence the lemma follows.
\end{proof}

%

\begin{lemma} \label{lemma:estimate:overlapping:indexes}
Let $0<r<R$. Then there is a constant $C_{R,r}>1$ such that
\begin{equation}\label{estimate:overlapping:indexes}
N_R(\{z_j\}_{j\ge1}) \le C_{R,r}\,N_r(\{z_j\}_{j\ge1}),
\end{equation}
for every sequence $\{z_j\}_{j\ge1}$ of different points in $\C$.
\end{lemma}

\begin{proof}
Let $\{z_j\}_{j\ge1}$ be a sequence of different points in $\C$ such that $N_r(\{z_j\}_{j\ge1})<\infty$.
By \eqref{eq:rhoinballs}, $c^{-1}_{R}\rho(z)\le\rho(w)\le c_{R}\,\rho(z)$,
 for every $z\in\C$ and $w\in D^{R}(z)$. 
 Then it is clear that $\cup_{z\in D^R(z_j)}D^r(z_j)\subset D^{c_R(r+R)}(z)$, for every $z\in\C$, and so
$$
 \#\{j\ge1\,:\,z\in D^R(z_j)\}\,r^2c_R^{-2}\rho(z)^2\le \sum_{z\in D^R(z_j)}r^2\rho(z_j)^2\le N_r(\{z_j\}_{j\ge1})\, c_R^2(r+R)^2\rho(z)^2.
$$	
Since 
$N_R(\{z_j\}_{j\ge1})=
	\sup_{z\in\C}\#\{j\ge1\,:\,z\in D^R(z_j)\}$,
we conclude that the constant $C_{R,r}=c_R^4(1+R/r)^2$ satisfies \eqref{estimate:overlapping:indexes}.
\end{proof}

\begin{lemma} \label{lemma:estimate:Ipr}
	$$
	\sup_{\zeta\in\C}\int_{\C}\bigg(\int_{D^r(\zeta)}
	\frac{\d\sigma(w)}{\exp(\d_\phi(z,w)^{\varepsilon})}
	\bigg)^p\d\sigma(z)<\infty\qquad(\varepsilon,p,r>0).
	$$
\end{lemma}

\begin{proof}
	For every $R>0$, we split the statement's integral on $\C$ into the corresponding integrals on $D^R(\zeta)$ and $D^R(\zeta)^c$, which we denote by $I_R(\zeta)$ and $J_R(\zeta)$, respectively. 
	Then \eqref{eq:rhoinballs} shows that
	$$
	\sup_{\zeta\in\C}I_R(\zeta)\le \sup_{\zeta\in\C} \sigma(D^R(\zeta))\,\sigma(D^r(\zeta))^p<\infty,
	\quad\mbox{for every $R>0$.}
	$$
	Now let us estimate $J_R(\zeta)$.
	Let $\zeta\in\C$, $z\in D^R(\zeta)^c$ and $w\in D^r(\zeta)$. If $R\ge r$ then $D^R(\zeta)^c\subset D^r(\zeta)^c$, and so \eqref{dist:estimate2} shows that 
	$\d_\phi(z,\zeta)\ge c_r^{-1}R^{\delta}$.
	On the other hand, \eqref{dist:estimate1} implies that  
	$\d_\phi(w,\zeta)\le c_r r$. 
	Thus if $R\ge(2c_r^2r)^{1/\delta}$ then 
	$\d_\phi(w,\zeta)\le\frac12\,\d_\phi(z,\zeta)$,
	and so 
	$\d_\phi(z,w)\ge\d_\phi(z,\zeta)-\d_\phi(w,\zeta)
	\ge\frac12\d_\phi(z,\zeta)$.
	Therefore
	$$
	J_R(\zeta)\le\sigma(D^r(\zeta))^p\int_{D^R(\zeta)^c}
	\frac{\d\sigma(z)}
	{\exp(\frac{p}{2^{\varepsilon}}\d_\phi(z,\zeta)^\varepsilon)}.
	$$             
	Now if $R\ge(4c_rp^{-2/\varepsilon})^{1/\delta}$ and $z\in D^R(\zeta)^c$ then 
	$\d_\phi(z,\zeta)^{\varepsilon/2}\ge (c_r^{-1}R^{\delta})^{\varepsilon/2}\ge \frac{2^{\varepsilon}}p$ and so   
	$\frac{p}{2^{\varepsilon}}\d_\phi(z,\zeta)^\varepsilon\ge  \d_\phi(z,\zeta)^{\varepsilon/2}$. Hence, for $R>0$ large enough, we obtain that    
	$$
	\sup_{\zeta\in\C} J_R(\zeta) \le \sup_{\zeta\in\C} \sigma(D^r(\zeta))^p\int_{D^R(\zeta)^c}
	\frac{\d\sigma(z)}
	{\exp(\d_\phi(z,\zeta)^{\varepsilon/2})}<\infty,
	$$          
	by \eqref{eq:rhoinballs} and \lemmaref{lemma:integral:estimate}. 
\end{proof}

\begin{lemma} \label{partition:lemma1}
	For $R>0$ and any finite sequence $\{z_j\}_{j=1}^n$ of different points in $\C$, let
	\begin{equation*}
		M_R(\{z_j\}_{j=1}^n):=\max_{1\le j\le n} 
		\#\{\,k\in\{1,\dots,n\}\,:\,|z_j-z_k|< R\min(\rho(z_j),\rho(z_k))\,\}.
	\end{equation*}
	Then $\{z_j\}_{j=1}^n$ can be partitioned into no more than $M_R(\{z_j\}_{j=1}^n)$ subsequences such that any different points $z_j$ and $z_k$ in the same subsequence satisfy either $z_j\not\in D^R(z_k)$ or $z_k\not\in D^R(z_j)$, that is,
	$|z_j-z_k|\ge R\min(\rho(z_j),\rho(z_k))$.
\end{lemma}

\begin{proof}
	We proceed by induction on $N=M_R(\{z_j\}_{j=1}^n)$.
	
	If $N=1$ then $|z_j-z_k|\ge R\min(\rho(z_j),\rho(z_k))$,
	for $1\le j<k\le n$, and there is nothing to prove.
	
	Let $N>1$. Then we may split $\{z_j\}_{j=1}^n$ into two subsequences $\{z'_j\}_{j=1}^{n'}$ and $\{z''_j\}_{j=1}^{n''}$ satisfying the following two conditions:
	\begin{enumerate}
		\item [(i)] $|z'_j-z'_k|\ge R\min(\rho(z'_j),\rho(z'_k))$,
		for $1\le j<k\le n'$. 
		\item [(ii)] For any $1\le j\le n''$ there is $1\le k\le n'$ so that $|z''_j-z'_k|< R\min(\rho(z''_j),\rho(z'_k))$.
	\end{enumerate}
	Namely, the points $z'_j$, $1\le j\le n'$,  can be inductively selected as follows:
	
	Let $z'_1$ be a point in $\{z_j\}_{j=1}^n$ such that
	$\rho(z'_1)=\min\{\rho(z_j)\,:\,1\le j\le n\,\}$. 
	
	Assume that $z'_{k-1}$ has been picked. 
	If $D^R(z'_1)\cup \dots \cup D^R(z'_{k-1})$ contains all the points from the sequence $\{z_j\}_{j=1}^n$, let $n'=k-1$ and stop the process of selection.
	Otherwise, pick a point $z'_k$ in $\{z_j\}_{j=1}^n\cap D^R(z'_1)^c\cap\dots\cap D^R(z'_{k-1})^c$ such that 
	$$
	\rho(z'_k)=\min\{\rho(z_j)\,:\,z_j\in D^R(z'_1)^c\cap\dots\cap D^R(z'_{k-1})^c\}.
	$$
	Then it is clear that the selected points $z'_j$, $1\le j\le n'$, satisfy (i), and, if $z''_j$, $1\le j\le n''$, are the non-selected points in $\{z_j\}_{j=1}^n$, (ii) also holds.

	Therefore $M_R(\{z''_j\}_{j=1}^{n''})\le N-1$ and so the induction hypothesis shows that we can partition $\{z''_j\}_{j=1}^{n''}$ into no more than $M_R(\{z''_j\}_{j=1}^{n''})$ subsequences satisfying the separation property of the statement. Hence the proof is complete.
\end{proof}

\begin{lemma} \label{partition:lemma2}
	Let $r\in(0,1)$ and $R>1$.
	Then 
	$M_R(\{z_j\}_{j=1}^n)\le 6^2R^4r^{-2}N_r(\{z_j\}_{j=1}^n)$,
	for every finite sequence $\{z_j\}_{j=1}^n$ in $\C$.
	Recall that $N_r(\{z_j\}_{j=1}^n)$ is defined by \eqref{overlapping:index}.
\end{lemma}

\begin{proof}
	Let $\{z_j\}_{j=1}^n$ be a finite sequence in $\C$, and, for  $1\le j\le n$, let
	$$
	A_j:=\{\,k\in\{1,\dots,n \}\,:\,|z_j-z_k|< R\min(\rho(z_j),\rho(z_k))\,\}
	\quad\mbox{and}\quad M_j:=\# A_j.
	$$
	If $k\in A_j$ then 
	$|\rho(z_j)-\rho(z_k)|\le|z_j-z_k|< R\min(\rho(z_j),\rho(z_k))$,
	by \eqref{rho:Lipschitz}, and so
	$\rho(z_j)\le 2R\rho(z_k)$ and $\rho(z_k)\le 2R\rho(z_j)$,
	since $R>1$. Now if $z\in D^r(z_k)$, for some $k\in A_j$, then 
	$|z-z_j|\le|z-z_k|+|z_k-z_j|<r\rho(z_k)+R\rho(z_j)\le 3R\rho(z_j)$,
	since $r<1$. Thus  $\cup_{k\in A_j}D^r(z_k)\subset D^{3R}(z_j)$.
	Therefore
	$$
	M_j\,\frac{r^2}{2^2R^2}\,\rho(z_j)^2\le
	r^2\sum_{k\in A_j}\rho(z_k)^2 \le 
	(3R)^2\rho(z_j)^2\,N_r,
	$$
	where $N_r=N_r(\{z_j\}_{j=1}^n)$. Hence 
	$M_R(\{z_j\}_{j=1}^n)=\max_{1\le j\le n}M_j\le 6^2\, R^4r^{-2}N_r$.
\end{proof}

\begin{lemma} \label{operator:A:lemma}
	Let $\{e_j\}_{j\ge1}$ be an orthonormal basis of $\gFockh$.
	Let $r>0$ and let $\{z_j\}_{j=1}^n$ be a finite sequence in
	$\C$.  Then
	\begin{equation}\label{operator:S:definition}
		Sf:=\sum_{j=1}^n\langle f, e_j\rangle_{\phi}\,K_{2,z_j}
		\qquad(f\in\gFockh),
	\end{equation}
	is a bounded linear operator on $\gFockh$ such that
	$\|S\|_{\gFockh\to\gFockh}\le C_r\,N_r(\{z_j\}_{j=1}^n)^{1/2}$, where $C_r>0$ is a constant that only depends on $r$ and $N_r(\{z_j\}_{j=1}^n)$ is defined by~{\eqref{overlapping:index}.}
\end{lemma}

\begin{proof}
	We only have to prove that
	\begin{equation}\label{estimate:operator:norm}
		|\langle Sf,g\rangle_{\phi}|\le C_r\,N_r^{1/2}\,\|f\|_{2,\phi}\|g\|_{2,\phi},
		\quad\mbox{for every $f,g\in\gFockh$,} 
	\end{equation}
	where $C_r>0$ is a constant that only depends on $r$ and $N_r:=N_r(\{z_j\}_{j=1}^n)$ .
	
	Let $f,g\in\gFockh$. Then Cauchy-Schwarz inequality shows that
	$$
	|\langle Sf,g\rangle_{\phi}|=
	\bigg|\sum_{j=1}^n\langle f, e_j\rangle_{\phi}
	\langle K_{2,z_j},g\rangle_{\phi}\bigg|\le
	\|f\|_{2,\phi}\bigg(\sum_{j=1}^n
	|\langle K_{2,z_j},g\rangle_{\phi}|^2\bigg)^{1/2}. 
	$$
	By \eqref{prop:kerneldiagonal}, there is a constant $c>0$ such that $e^{\phi(z)}/(c\rho(z))\le\|K_z\|_{2,\phi}$, 
	for every $z\in\C$, and so
	$$
	|\langle Sf,g\rangle_{\phi}|\le c\,\|f\|_{2,\phi}
	\bigg(\sum_{j=1}^n\rho(z_j)^2|g(z_j)e^{-\phi(z_j)}|^2\bigg)^{1/2}.
	$$
	Now \lemmaref{lemma:estimates} and \eqref{eq:rhoinballs} show that there is a constant $C(r)>0$, which only depends on $r$, such that 
	$$
	\rho(z)^2|h(z)e^{-\phi(z)}|^2\le C(r)^2\int_{D^r(z)}|he^{-\phi}|^2\d A,
	$$
	for every $h\in\H(\C)$ and $z\in\C$. Therefore
	$$
	|\langle Sf,g\rangle_{\phi}|\le c\,C(r)\|f\|_{2,\phi}
	\bigg(\sum_{j=1}^n\int_{D^r(z_j)}|ge^{-\phi}|^2\d A\bigg)^{1/2}
	\le c\,C(r)N_r^{1/2}\|f\|_{2,\phi}\|g\|_{2,\phi}.
	$$
	Hence we conclude that $C_r=c\,C(r)$ satisfies \eqref{estimate:operator:norm}.
\end{proof}

\subsection{Proof of the equivalence of \eqref{sp:second}, \eqref{sp:third},  \eqref{sp:four} and \eqref{sp:fith} of \thmref{thm:Tzsp}}
It is clear that $\eqref{sp:second}\Rightarrow\eqref{sp:third}$, while 	$\eqref{sp:fith}\Rightarrow\eqref{sp:four}$ is a direct consequence of \lemmaref{lemma:carleson1}. So we only have to prove 
$\eqref{sp:third}\Rightarrow\eqref{sp:fith}$ and $\eqref{sp:four}\Rightarrow\eqref{sp:second}$.
\vspace*{3pt}
\imply{(\ref{sp:four})}{(\ref{sp:second})} Since 
\begin{equation}\label{monotonicity:average}
s^2\,\widehat{\mu}_s(z)\le r^2\,\widehat{\mu}_r(z),
\quad\mbox{for every $0<s<r$ and $z\in\C$,}	
\end{equation}
we may assume that there is $r_1\in(0,1/2)$ such that $\widehat{\mu}_r\in L^p(\C,\d\sigma)$, for every $r\in(0,r_1)$.
 Then \lemmaref{lemma:kappa:constant} shows that any $(r/4,\phi)$-lattice 
 $\{z_j\}_{j\ge1}$ with $r\in(0,r_1)$ satisfies 
 $$
 \frac{r^2}{32}\sum_{j=1}^{\infty}\widehat{\mu}_{r/4}(z_j)^p
 \le  4^p\,\sum_{j=1}^{\infty}
  \int_{D_j}\widehat{\mu}_r(w)^p\d\sigma(w)
 \le 4^p\,N_r
 \int_{\C} \widehat{\mu}_{r}(w)^p\d\sigma(w),
 $$
 where $D_j:=D^{r/4}(z_j)$ and $N_r:=N_{r/4}(\{z_j\}_{j\ge1})$.
 Therefore we have just proved that $r_0=r_1/4$ satisfies \eqref{sp:second}.
\vspace*{3pt}
\imply{(\ref{sp:third})}{(\ref{sp:fith})} 
The following proposition proves this implication in a quantitative way. 
\begin{proposition}
	For any $0<p<\infty$ and $r>0$ there is a constant 
	$C_{p,r}>0$ such that every $\mu\in\M$ and every $(r,\phi)$-lattice $\{z_j\}_{j\ge1}$ satisfy 
	$$
	\int_{\C}\tld{\mu}(z)^p\d\sigma(z) 
	\le C_{p,r}\,N_r(\{z_j\}_{j\ge1})^{(p\vee 1)-1}\,\sum_{j=1}^{\infty}\widehat{\mu}_r(z_j)^p.
	$$
\end{proposition}

\begin{proof} 
	Let $0<p<\infty$ and $r>0$. 
	Along this proof  $A\lesssim B$ means that $A\le C\,B$, where $C>0$ is a constant that only depends on $p$ and $r$. 
	 Let $\mu\in\M$ and let $\{z_j\}_{j\ge1}$ be an $(r,\phi)$-lattice. 
	First we are going to obtain a pointwise estimate of $\widetilde{\mu}$ in terms of the sequence $\{\widehat{\mu}_r(z_j)\}_{j\ge1}$. 
	Note that
	$$
	\widetilde{\mu}(z)=\int_{\C}|K_{2,z}e^{-\phi}|^2\d\mu
	\le\sum_{j=1}^{\infty}\int_{D^r(z_j)}|K_{2,z}e^{-\phi}|^2\d\mu.
	$$
	Then, by \lemmaref{lemma:estimates},
	$$
	|K_{2,z}(w)e^{-\phi(w)}|^2\lesssim \int_{D^r(w)}|K_{2,z}e^{-\phi}|^2\,
	\frac{\d A}{\rho^2}\qquad (w\in\C).
	$$
	Moreover, \eqref{rho:Lipschitz} shows that $A_j:=\cup_{w\in D^r(z_j)}D^r(w)\subset D^{r(2+r)}(z_j)$, while \eqref{eq:rhoinballs} implies that $\rho(z_j)\lesssim\inf_{\zeta\in A_j}\rho(\zeta)$.
	Thus
	$$
	|K_{2,z}(w)e^{-\phi(w)}|^2\lesssim\frac1{\rho(z_j)^2}
    \int_{D_j}|K_{2,z}e^{-\phi}|^2\d A
    \qquad(w\in D^r(z_j),\,j\ge1),
	$$
	where $D_j:=D^{r(2+r)}(z_j)$, and  therefore we get the pointwise estimate
	\begin{equation}\label{pointwise:estimate}
	\widetilde{\mu}(z)\lesssim \sum_{j=1}^{\infty}
	\widehat{\mu}_r(z_j) \int_{D_j}|K_{2,z}e^{-\phi}|^2
	\d A\qquad(z\in\C).
	\end{equation}
	
	If $0<p\le1$ then \eqref{pointwise:estimate} implies that 
	\begin{equation}\label{integral:estimate:p:less:than:1}
	\int_{\C}\widetilde{\mu}(z)^p\d\sigma(z)\lesssim \sum_{j=1}^{\infty}
	\widehat{\mu}_r(z_j)^p\, I_{p,r}(z_j),
	\end{equation}
	where
	$$
	I_{p,r}(\zeta):=\int_{\C}\bigg(\int_{D^{r(2+r)}(\zeta)}
	|K_{2,z}e^{-\phi}|^2 \d A\bigg)^p\d\sigma(z)
	\qquad(\zeta\in\C).
	$$
	
	On the other hand, if $1<p<\infty$ and $q$ is the conjugate exponent of $p$, then  \eqref{pointwise:estimate} and H\"{o}lder's inequality show that
	\begin{align*}
	\widetilde{\mu}(z)^p 
	&\lesssim \bigg(\sum_{j=1}^{\infty}\widehat{\mu}_r(z_j)^p
	                          \int_{D_j} |K_{2,z}e^{-\phi}|^2 \d A\bigg)
	\bigg(\sum_{j=1}^{\infty} 
	\int_{D_j} |K_{2,z}e^{-\phi}|^2 \d A\bigg)^{\frac{p}q}\\
	&\lesssim N_{r(2+r)}(\{z_j\}_{j\ge1})^{p-1}\,
	\sum_{j=1}^{\infty}\widehat{\mu}_r(z_j)^p
	\int_{D_j} |K_{2,z}e^{-\phi}|^2 \d A.
	\end{align*}
	By integrating the preceding estimate and applying \lemmaref{lemma:estimate:overlapping:indexes} we get
	\begin{equation}\label{integral:estimate:p:greater:than:1}
	\int_{\C}\widetilde{\mu}(z)^p\d\sigma(z)\lesssim 
	N_r(\{z_j\}_{j\ge1})^{p-1}\,
	\sum_{j=1}^{\infty}
	\widehat{\mu}_r(z_j)^p\, I_{1,r}(z_j).
	\end{equation}
	
	Since \eqref{integral:estimate:p:less:than:1} holds for $0<p\le1$, while \eqref{integral:estimate:p:greater:than:1} holds for $1<p<\infty$, the proof will be complete once we prove that  $\sup_{\zeta\in\C}I_{p,r}(\zeta)<\infty$.
	 Note that \eqref{prop:kerneldiagonal} and \eqref{eq:kernel1} imply the estimate
	$$
	|K_{2,z}(w)e^{-\phi(w)}| 
	\simeq \rho(z)|K_z(w)|e^{-\phi(z)-\phi(w)}
	\lesssim\frac1{\rho(w)\exp(\d_\phi(z,w)^{\varepsilon})}              \quad(z,w\in\C),              
	$$
	and so
	$$
	I_{p,r}(\zeta)\lesssim\int_{\C}\bigg(\int_{D^{r(2+r)}(\zeta)}
	\frac{\d\sigma(w)}{\exp(\d_\phi(z,w)^{\varepsilon})}
	\bigg)^p\d\sigma(z)
	\qquad(\zeta\in\C).
	$$
	Then \lemmaref{lemma:estimate:Ipr} shows that	$\sup_{\zeta\in\C}I_{p,r}(\zeta)<\infty$, and the proof is finished.
\end{proof}

\subsection{End of the proof of \thmref{thm:Tzsp} for
	$1\le p<\infty$}
\imply{(\ref{sp:first})}{(\ref{sp:fith})} 
It directly follows from 
\lemmaref{lemma:aver:in:Lp:Tmu:in:Sp}(a).
	\imply{(\ref{sp:four})}{(\ref{sp:first})}
Assume that $\widehat{\mu}_r \in L^p(\C,\d\sigma)$. By 
\eqref{monotonicity:average} we may assume that $r\in(0,1/2)$. 
Then, by \lemmaref{lemma:kappa:constant},
$r^2/32\le \sigma(D^{r/4}(z))$ and 
$\widehat{\mu}_{r/4}(z)\le 4\,\widehat{\mu}_r(w)$, for every $z\in\C$ and $w\in D^{r/4}(z)$, 
so
\begin{equation}\label{estimate:average:transform:by:its:Lp:norm}
\widehat{\mu}_{r/4}(z) \lesssim
\biggl(\int_{D^{r/4}(z)}(\widehat{\mu}_r(w))^p \d\sigma(w)\biggr)^{1/p}
\le \|\widehat{\mu}_r\|_{L^p(\C,\d\sigma)} \qquad (z\in\C),
\end{equation}
and therefore Theorems~\ref{thm:carleson1} 
and~\ref{thm:Tzbounded} show that $\Tz$ is bounded on $\gFockh$.
 Thus it only remains to 
prove that 
$\tr{\Tz^p}\lesssim\|\widehat{\mu}_{r}\|^p_{L^p(\C,\d\sigma)}$.

Let $(e_n)_{n\ge1}$ be an orthonormal basis of $\gFockh$. Then 
\eqref{positiveness:Toeplitz:operator} and~{\lemmaref{lemma:mumean}(a)} imply that
	$$
	\psc{\Tz e_n, e_n}_\phi = 
	\int_\C \abs{e_n(z)}^2 \efh{z}\dist \mu(z) \lesssim \int_\C \widehat{\mu}_{r}(z)\abs{e_n(z)}^2 \efh{z}\dist A(z).
	$$
	Since $\abs{e_n}^2 \ee^{-2\phi} \d A$ is a probability measure on $\C$ and $p\ge1$, Jensen's inequality shows that
	$$
	\psc{\Tz e_n, e_n}_\phi ^p 
	\lesssim  \int_\C  \widehat{\mu}_{r}(z)^p  \abs{e_n(z)}^2 \efh{z} \d A(z).
	 $$
	Finally, by summing up in the previous estimate and applying the monotone convergence theorem and \eqref{prop:kerneldiagonal}, we conclude that
\begin{eqnarray*}
\tr{\Tz^p} 
& \lesssim&   \int_\C  \widehat{\mu}_{r}(z)^p 
\sum_{n=1}^\infty          
 \abs{e_n(z)}^2 \efh{z} \d A(z)\\
&=&  \int_\C  \widehat{\mu}_{r}(z)^p  \gfhnorm{K_z}^2 \efh{z} 
\d A(z) 
\lesssim \int_\C  \widehat{\mu}_{r}(z)^p  \d \sigma(z).
\end{eqnarray*}

\subsection{End of the proof of \thmref{thm:Tzsp} for
	$0<p<1$}
	
\noindent
This proof is more involved than the one of the case
 $1\le p<\infty$. 
It will be done by proving the chain of implications
$\mbox{(\ref{sp:fith})}\Rightarrow\mbox{(\ref{sp:first})}
\Rightarrow\mbox{(\ref{sp:second})}$.
\vspace*{3pt}
\imply{(\ref{sp:fith})}{(\ref{sp:first})} 
Assume that $\tld{\mu}\in L^p(\C,\! \d \sigma)$. Then, by  
\lemmaref{lemma:carleson1}, there is $r\in(0,1/2)$ such that  
$\widehat{\mu}_r \in L^p(\C,\d\sigma)$, so \eqref{estimate:average:transform:by:its:Lp:norm} holds. Therefore, by~\thmref{thm:carleson1}, $\mu$ is a $\phi$-Fock-Carleson measure, 
and hence~\thmref{thm:Tzbounded}
implies that $\Tz$ is bounded on $\gFockh$.
 Finally,  the membership of $\Tz$  in the 
Schatten class $\S_p(\gFockh)$ follows from 
\lemmaref{lemma:aver:in:Lp:Tmu:in:Sp}(b).
\vspace*{3pt}

\noindent
$\mbox{\eqref{sp:first}}\Rightarrow\mbox{\eqref{sp:second}:}$ 
This is the most difficult part of the proof of \thmref{thm:Tzsp}.
The following proposition proves this implication in a quantitative way. 

\begin{proposition} \label{1:implies:2:p:smaller:1}
	Let $p\in(0,1)$ and let $r_0\in(0,1)$ satisfying \eqref{eq:kernel2}.
	 Then for any $r\in(0,r_0)$ and for any integer $N>0$ there is a constant $C=C_{p,r,N}>0$ so that
\begin{equation}\label{estimate:sums:means}
\sum_{j=1}^{\infty} \widehat{\mu}_{r}(z_j)^p 
\le C\,\|T_{\mu}\|_{\S_p}^p,
\end{equation}
	for every $\phi$-Fock-Carleson measure $\mu$ and
	for every $(r,\phi)$-lattice $\{z_j\}_{j\ge1}$ such that $N_r(\{z_j\}_{j\ge1})\le N$.
\end{proposition}

The idea of the proof of \propref{1:implies:2:p:smaller:1} has its origins in the work of S. Semmes~{\cite{semmes}} and D. Luecking~{\cite{luecking}}. That result will easily deduced from two simple separation lemmas (Lemmas~{\ref{partition:lemma1}} and~{\ref{partition:lemma2}})  and the following key lemma.

\begin{lemma} \label{lemma:1:implies:3:p:smaller:1}
	Let $p\in(0,1)$ and let $r_0\in(0,1)$ satisfying \eqref{eq:kernel2}.
	Then for every $r\in(0,r_0)$ and for every integer $N>0$ there are constants $C=C(p,r,N)>0$ and $R=R(p,r,N)>1$ such that
	\begin{equation}\label{estimate:sums:means}
	\sum_{j=1}^n \widehat{\mu}_{r}(z_j)^p 
	\le C\,\|T_{\mu}\|_{\S_p}^p,
	\end{equation}
	for every $\phi$-Fock-Carleson measure $\mu$ and
	for every finite sequence $\{z_j\}_{j=1}^n$ in $\C$ so that $N_r(\{z_j\}_{j\ge1}^n)\le N$ and  
	$|z_j-z_k|\ge R\,\min(\rho(z_j),\rho(z_k))$, for $1\le j<k\le n$.
\end{lemma}

\begin{proof} 
	Let $r\in(0, r_0)$ and let $N$ be a positive integer. Let $\mu$ be a $\phi$-Fock-Carleson measure so that $\Tz\in\S_p(\gFockh)$.
 Let $\{z_j\}_{j=1}^n$ be a finite sequence  in $\C$ such that $N_r(\{z_j\}_{j\ge1}^n)\le N$ and  
$|z_j-z_k|\ge R\min(\rho(z_j),\rho(z_k))$, for $1\le j<k\le n$.
We are going to prove that there are constants $C>0$ and $R>1$, which only depend on $p$, $r$ and $N$, so that~{\eqref{estimate:sums:means}} holds. Along this proof  $A\lesssim B$ means that $A\le C\,B$, where $C>0$ is a constant that only depends on $p$, $r$ and $N$. 

Let us consider the measure $\nu$ defined by
$$
\d\nu:=\left(\sum_{j=1}^n\chi_{D^r(z_j)}\right)\d\mu.
$$
Then $\nu\in\M$ and $\nu\le N\,\mu$. It directly follows that $\nu$ is a $\phi$-Fock-Carleson measure, because so is $\mu$, and we have that $0\le T_{\nu}\le N\,T_{\mu}$, by~\eqref{positiveness:Toeplitz:operator}. 
Since $0<p<1$, the L\"{o}wner-Heinz inequality (see~\cite{lowner,heinz,pedersen}) shows that $0\le T_{\nu}^p\le N^p\,T_{\mu}^p$, and therefore 
$\|T_\nu\|_{\S_p}\le N\,\|T_{\mu}\|_{\S_p}$.

Let $T:=S^{*}T_{\nu}S$, where $S$ is the bounded linear operator on $\gFockh$ defined by~{\eqref{operator:S:definition}.} 
Namely, $T$ is the bounded linear operator on $\gFockh$ given by
\begin{equation}\label{inner:product:def:operator:T}
\langle Tf, g\rangle_{\phi} =
\langle T_{\nu}(Sf), Sg\rangle_{\phi}\qquad(f,g\in\gFockh).
\end{equation}
It follows from~{\cite[Theorem 1.6]{simon}} and \lemmaref{operator:A:lemma} that 
$$
\|T\|_{\S_p}
\le \|S\|^2_{\gFockh\to\gFockh}\,\|T_\nu\|_{\S_p}
\lesssim N\,\|T_\mu\|_{\S_p}, 
$$
 
Note that $Se_j=0$, for every $j>n$, and so, by~\eqref{inner:product:def:operator:T}, we have that
$$
Tf=\sum_{j,k=1}^n
\langle T_{\nu}K_{2,z_j},K_{2,z_k}\rangle_{\phi}\,
\langle f, e_j\rangle_{\phi}\, e_k\qquad(f\in\gFockh).
$$
Thus we decompose $T$ as the sum of its ``diagonal'' part
$$
Df:=\sum_{j=1}^n
\langle T_{\nu}K_{2,z_j},K_{2,z_j}\rangle_{\phi}\,
\langle f, e_j\rangle_{\phi}\, e_j\qquad(f\in\gFockh).
$$
and its ``non-diagonal'' part $E:=T-D$. Therefore, by Rotfel'd inequality \eqref{rotfel'd:inequality}, 
\begin{equation}\label{estimate:diagonal:nondiagonal}
N^{p}\,\|T_\mu\|_{\S_p}^p\gtrsim\|T\|_{\S_p}^p\ge
\|D\|_{\S_p}^p-\|E\|_{\S_p}^p,
\end{equation}
and we are going to get a lower estimate of $\|D\|_{\S_p}^p$
and an upper estimate of $\|E\|_{\S_p}^p$. 
By \eqref{positiveness:Toeplitz:operator} and \lemmaref{lemma:carleson1},
 we have that
$$
\langle T_{\nu}K_{2,z_j},K_{2,z_j}\rangle_{\phi}
= \int_{\C} \abs{K_{2,z_j}(w)\ef{w}}^2\d\nu(w)
\gtrsim \widehat{\mu}_r (z_j),
$$ 
so the estimate of $\|D\|_{\S_p}^p$ easily follows:
\begin{equation}\label{estimate:D}
\|D\|_{\S_p}^p=
\sum_{j=1}^n \langle T_{\nu}K_{2,z_j},K_{2,z_j}\rangle_{\phi}^p 
\gtrsim\sum_{j=1}^n \widehat{\mu}_r (z_j)^p.
\end{equation}
The estimate of $\|E\|_{\S_p}^p$ is more involved. 
First, by \cite[Proposition 1.29]{zhu2007operator} we have
$$
\|E\|_{\S_p}^p
\le \sum_{j,k=1}^{\infty}|\langle E e_j, e_k\rangle_{\phi}|^p
=  \sum_{\substack{j,k=1\\ j\ne k}}^n 
|\langle T_{\nu} K_{2,z_j}, K_{2,z_k}\rangle_{\phi}|^p
\le  \sum_{\substack{j,k=1\\ j\ne k}}^n  S_{j,k}^p,    
$$
where 
\begin{equation}\label{definition:Sjk}
S_{j,k}:=\sum_{\ell=1}^n
\int_{D^r(z_{\ell})} |K_{2,z_j}(w) K_{2,z_k}(w)| e^{-2\phi(w)} \d \mu(w).
\end{equation}
We want to estimate $S_{j,k}$ from above by a small constant times $\sum_{j=1}^n \widehat{\mu}_r (z_j)^p$. We start on by applying \eqref{eq:kernel1} and \eqref{prop:kerneldiagonal} to obtain
$$
|K_{2,z}(w)|e^{-\phi(w)}
\lesssim \frac1{\rho(w)}\,
    	\frac1{\exp(\dist_{\phi}(z,w)^{\varepsilon})}
    	\qquad(z,w\in\C),
$$  
and, in particular,
$$
|K_{2,z}(w)|e^{-\phi(w)}
\lesssim \frac1{\rho(w)^{\frac12}}\,
\frac{|K_{2,z}(w)e^{-\phi(w)}|^{\frac12}}{\exp(\frac12\d_{\phi}(z,w)^{\varepsilon})}\qquad(z,w\in\C).
$$ 
Therefore
\begin{equation}\label{estimate1:Ijkl}
|K_{2,z_j}(w) K_{2,z_k}(w)|e^{-2\phi(w)}
\lesssim \frac1{\rho(w)}\,
    \frac{|K_{2,z_j}(w)e^{-\phi(w)}|^{\frac12}
    	  |K_{2,z_k}(w)e^{-\phi(w)}|^{\frac12}}
{\exp(\frac12\d_{\phi}(z_j,w)^{\varepsilon}+
	  \frac12\d_{\phi}(z_k,w)^{\varepsilon})}.      
\end{equation}
Now we are going to prove that, for $1\le j,k,\ell\le n$, $j\ne k$, we have 
\begin{equation}\label{estimate2:Ijkl}
\inf_{w\in D^r(z_{\ell})}
\exp(\tfrac12\dist_{\phi}(z_j,w)^{\varepsilon}+\tfrac12\dist_{\phi}(z_k,w)^{\varepsilon})
\ge c_{\varepsilon,r}(R)\to\infty,
\quad\mbox{as $R\to\infty$,} 
\end{equation}
where $c_{\varepsilon,r}(R)>0$ is a constant which only depends on $\varepsilon$, $r$ and $R$.

\noindent
If $1\le j,k,\ell\le n$, $j\ne k$, then either $\ell\ne j$ or $\ell\ne k$. 
We may assume that $\ell\ne j$, since otherwise we may  replace $j$ by $k$.
 Then either $z_{\ell}\in D^R(z_j)^c\subset D(z_j)^c$ or 
  $z_j\in D^R(z_{\ell})^c\subset D(z_{\ell})^c$,
   so, by \eqref{dist:estimate1} and \eqref{dist:estimate2},
   we have that
  $$
 \dist_{\phi}(z_j,w) \ge
 \dist_{\phi}(z_j,z_{\ell})-\dist_{\phi}(w,z_{\ell}) \ge
 c_1^{-1}R^{\delta}-c_1r,
 \quad\mbox{for every $w\in D^r(z_{\ell})$.}  
 $$
  It follows that 
 $c_{\varepsilon,r}(R):=
 \exp(\frac12(c_1^{-1}R^{\delta}-c_1r)^{\varepsilon})$
 satisfies \eqref{estimate2:Ijkl}.
 
Then~{\eqref{estimate1:Ijkl}} and~{\eqref{estimate2:Ijkl}}
 show that
 \begin{equation}\label{estimate3:Ijkl}
 |K_{2,z_j}(w) K_{2,z_k}(w)|e^{-2\phi(w)}
 \lesssim \frac1{c_{\varepsilon,r}(R)}
 \biggl|\frac{K_{2,z_j}(w)e^{-\phi(w)}}
                        {\rho(w)}\biggr|^{\frac12}
 \biggl|\frac{|K_{2,z_k}(w)e^{-\phi(w)}}
                        {\rho(w)}\biggr|^{\frac12},
\end{equation}  
 for every $w\in D^r(z_{\ell})$. Recall that, by \lemmaref{lemma:rhoinballs}, there is a constant $c\ge1$ so that
 \begin{equation}\label{rho:in:unit:ball}
 c^{-1}\,\rho(z)\le\rho(w)\le c\,\rho(z),\quad
 \mbox{for every $z\in\C$ and $w\in D(z)$.} 
 \end{equation}
 Then \lemmaref{lemma:estimates} and~{\eqref{rho:in:unit:ball}} imply that 
$$
\biggl|\frac{K_{2,z}(w)e^{-\phi(w)}}
                      {\rho(w)}\biggr|^{\frac{p}2}
\lesssim  \frac1{\rho(\zeta)^{\frac{p}2+2}}
     \int_{D^r(w)}|K_{2,z}\,e^{-\phi}|^{\frac{p}2}\d A
\quad(z,\zeta\in\C,\,w\in D^r(\zeta)).     
$$
Moreover, by {\eqref{rho:in:unit:ball}}, 
$D^r(w)\subset D^{2cr}(z_{\ell})$, for every $w\in D^r(z_{\ell})$, and so
\begin{equation}\label{estimate4:Ijkl}
\biggl|\frac{K_{2,z_j}(w)e^{-\phi(w)}}
            {\rho(w)}\biggr|^{\frac12}
\lesssim\frac1{\rho(z_{\ell})^{\frac12+\frac2p}}
        \,I_{j,\ell}^{\frac1p},\quad
(1\le j,\ell\le n,\,w\in D^r(z_{\ell})),
\end{equation}
where 
$$
I_{j,\ell}:=
\int_{D^{2cr}(z_{\ell})}|K_{2,z_j}\,e^{-\phi}|^{\frac{p}2}\d A.
$$ 
Therefore it follows from \eqref{definition:Sjk}, \eqref{estimate3:Ijkl}, \eqref{rho:in:unit:ball} and \eqref{estimate4:Ijkl} that 
$$
S_{j,k}
\lesssim \frac1{c_{\varepsilon,r}(R)}
     \sum_{\ell=1}^n
     \frac{(I_{j,\ell}\,I_{k,\ell})^{\frac1p}}
          {\rho(z_{\ell})^{\frac4{p}+1}}\,
      \mu(D^r(z_{\ell}))
\lesssim \frac1{c_{\varepsilon,r}(R)}
      \sum_{\ell=1}^n
      \frac{(I_{j,\ell}\,I_{k,\ell})^{\frac1p}}
           {\rho(z_{\ell})^{\frac4{p}-1}}\,
       \widehat{\mu}_r(z_{\ell}),
$$ 
 and, since $0<p<1$, we get that
\begin{eqnarray*}
\|E\|_{\S_p}^p
&\le& \sum_{\substack{j,k=1\\ j\ne k}}^n S_{j,k}^p
\lesssim \frac1{c_{\varepsilon,r}(R)^p}
\sum_{\substack{j,k=1\\ j\ne k} }^n
\sum_{\ell=1}^n 
     \frac{I_{j,\ell}\,I_{k,\ell}}{\rho(z_{\ell})^{4-p}}\,
     \widehat{\mu}_r(z_{\ell})^p\\ 
&\le& \frac1{c_{\varepsilon,r}(R)^p}
  \sum_{\ell=1}^n
  \frac{\widehat{\mu}_r(z_{\ell})^p}{\rho(z_{\ell})^{4-p}}
  \biggl(\sum_{j=1}^nI_{j,\ell}\biggr)^2. 
\end{eqnarray*}
Now
$$
\sum_{j=1}^nI_{j,\ell} = 
\int_{D^{2cr}(z_{\ell})} 
\biggl(\sum_{j=1}^n|K_{2,z_j}(w)|^{\frac{p}2}\biggr)
                            e^{-\frac{p}2\phi(w)}\d A(w).
$$
Moreover, \eqref{prop:kerneldiagonal}, \lemmaref{lemma:estimates}, \eqref{eq:rhoinballs},  \eqref{eq:kernel1} and \cite[Lemma 2.7]{marzo2009pointwise} show that
\begin{eqnarray*}
\sum_{j=1}^n|K_{2,z_j}(w)|^{\frac{p}2}
&\lesssim& \sum_{j=1}^n \rho(z_j)^{\frac{p}2}
           |K_{w}(z_j)e^{-\phi(z_j)}|^{\frac{p}2}\\
&\lesssim& \sum_{j=1}^n \rho(z_j)^{\frac{p}2}
           \int_{D^r(z_j)}
          |K_{w}(z)e^{-\phi(z)}|^{\frac{p}2}\d\sigma(z)\\
&\lesssim& \sum_{j=1}^n \int_{D^r(z_j)}
             |K_{w}(z)e^{-\phi(z)}|^{\frac{p}2}
                          \rho(z)^{\frac{p}2}\d\sigma(z)\\
&\le& N \int_{\C}|K_{w}(z)e^{-\phi(z)}|^{\frac{p}2}
          \rho(z)^{\frac{p}2}\d\sigma(z)\\
&\lesssim& 
         N\, \frac{e^{\frac{p}2\phi(w)}}{\rho(w)^{\frac{p}2}}
         \int_{\C}\frac{\d\sigma(z)}
          {\exp(\frac{p}2\d_{\phi}(z,w)^{\varepsilon})} \\
&\lesssim& 
N\,\frac{e^{\frac{p}2\phi(w)}}{\rho(w)^{\frac{p}2}}
\int_{\C}\frac{\d\sigma(z)}
{\exp(\d_{\phi}(z,w)^{\frac{\varepsilon}2})}           
 \lesssim N\,\frac{e^{\frac{p}2\phi(w)}}{\rho(w)^{\frac{p}2}}.       
\end{eqnarray*}
It follows, by \eqref{eq:rhoinballs}, that
$$
\sum_{j=1}^nI_{j,\ell} \lesssim
N \int_{D^{2cr}(z_{\ell})} \frac{\d A(w)}{\rho(w)^{\frac{p}2}}
\lesssim N \rho(z_{\ell})^{2-\frac{p}2}
$$
and so 
\begin{equation}\label{estimate:E}
\|E\|_{\S_p}^p
\lesssim \frac{N^2}{c_{\varepsilon,r}(R)^p}
\sum_{\ell=1}^n
\frac{\widehat{\mu}_r(z_{\ell})^p}{\rho(z_{\ell})^{4-p}}
\biggl(\sum_{j=1}^n I_{j,\ell}\biggr)^2 
\lesssim \frac{N^2}{c_{\varepsilon,r}(R)^p}
\sum_{\ell=1}^n\widehat{\mu}_r(z_{\ell})^p.
\end{equation}
Hence \eqref{estimate:diagonal:nondiagonal},
 \eqref{estimate:D} and \eqref{estimate:E} show that there
  are constants $C_1,C_2>0$ which only depend on $p$ and $r$ 
  and satisfy
 $$
 \sum_{\ell=1}^n\widehat{\mu}_r(z_{\ell})^p\le 
 N^p\biggl(C_1-\frac{C_2 N^2}{c_{\varepsilon,r}(R)}\biggr)^{-1}
 \|T_{\mu}\|_{\S_p}^p
 $$
 Since $c_{\varepsilon,r}(R)\to\infty$, as $R\to\infty$, it is clear that there is $R=R(p,r,N)>1$ such that $C=2N^p/C_1$
 satisfies \eqref{estimate:sums:means}, and the proof is complete.  
\end{proof}
 
 \begin{proof}[Proof of \propref {1:implies:2:p:smaller:1}] 
 	Let $r\in(0, r_0)$ and let $N$ be a positive integer. Let $\mu$ be a $\phi$-Fock-Carleson measure and let
 	$\{z_j\}_{j\ge1}$ be an $(r,\phi)$-lattice
 	 such that $N_r(\{z_j\}_{j\ge1})\le N$. Fix a positive integer $n$. 
 	  We want to prove that 
 $$
 \sum_{j=1}^n \widehat{\mu}_{r}(z_j)^p 
 \le C\,\|T_{\mu}\|_{\S_p}^p,
 $$
 for some constant $C>0$ only depending on $p$, $r$ and $N$.	
By Lemmas~{\ref{partition:lemma1}} and~{\ref{partition:lemma2}}, for every $R>1$, the finite sequence $\{z_j\}_{j=1}^n$ can be partitioned into no more than $6^2 R^4r^{-2}N$ subsequences such that any different points $z_j$ and $z_k$ in the same subsequence satisfy $|z_j-z_k|\ge R\min(\rho(z_j),\rho(z_k))$.
Therefore \lemmaref{lemma:1:implies:3:p:smaller:1} shows that 
$$
\sum_{j=1}^n \widehat{\mu}_{r}(z_j)^p 
\le 6^2 R^4r^{-2}NC\,\|T_{\mu}\|_{\S_p}^p,
$$
for constants $C>0$ and $R>1$ only depending on $p$, $r$ and $N$.
Hence the proof is finished.
\end{proof}

\section*{Acknowledgements}

The starting point of this paper was the Master's thesis of the first-named author
under the supervision of Dr.~Joaquim Ortega-Cerd\`a, to whom he  would like to express his deep gratitude for his valuable guidance, assistance and advice. He also wishes to thank to Dr.~Jordi Pau for his useful comments, suggestions and his encouragement in carrying out this project work, mostly in the final state. 

The second-named author would like to thank to Joaquim Ortega-Cerd\`a and Jordi Pau for several informal discussions on the subject of this paper, and specially to Joan F\`{a}brega who carefully read several versions of the paper and made a number of useful suggestions that improve the quality of the paper.


\bibliographystyle{amsplain}

\end{document}